\documentclass[10pt,leqno]{article}
\usepackage{amsmath,amssymb,amscd,latexsym,bm}

\usepackage{xcolor}

\usepackage{changebar}

\newcommand{\dis}{\displaystyle}
\usepackage[all]{xy}
\newcommand{\A}{{\mathbb{A}}}
\newcommand{\C}{{\mathbb{C}}}
\newcommand{\D}{\mathbb{D}}

\newcommand{\Ge}{\mathbb{G}}
\newcommand{\N}{\mathbb{N}}

\newcommand{\R}{{\mathbb{R}}}
\newcommand{\Sa}{\mathbb{S}}
\newcommand{\Z}{{\mathbb{Z}}}

\newcommand{\hZ}{\hat{\Z}}

\newcommand{\cV}{\check{V}}
\newcommand{\cW}{\check{W}}

\newcommand{\alg}{\mathrm{alg}}
\newcommand{\Bij}{\mathrm{Bij}}
\newcommand{\car}{\mathrm{char}}

\newcommand{\et}{\mathrm{\acute{e}t}}
\newcommand{\ev}{\mathrm{ev}}

\newcommand{\id}{\mathrm{id}}

\newcommand{\uMod}{\mathbf{Mod}}

\newcommand{\Mor}{\mathrm{Mor}}

\newcommand{\red}{\mathrm{red}}
\newcommand{\rk}{\mathrm{rk}\,}

\newcommand{\spec}{\mathrm{spec}\,}

\newcommand{\Aut}{\mathrm{Aut}}
\newcommand{\uAut}{\underline{\Aut}}
\newcommand{\uFCov}{\mathbf{FCov}}
\newcommand{\uFSet}{\mathbf{FSet}}

\newcommand{\diag}{\mathrm{diag}}
\newcommand{\End}{\mathrm{End}}
\newcommand{\Ext}{\mathrm{Ext}}

\newcommand{\Fl}{\mathrm{Fl}}
\newcommand{\uFl}{\mathbf{Fl}}
\newcommand{\uIFl}{\mathbf{IFl}}
\newcommand{\uFFl}{\mathbf{FFl}}

\newcommand{\GL}{\mathrm{GL}}
\newcommand{\uGL}{\underline{\GL}}

\newcommand{\Hom}{\mathrm{Hom}}
\newcommand{\uHom}{\underline{\Hom}}
\newcommand{\Imm}{\mathrm{Im}\,}
\newcommand{\imm}{\mathrm{im}\,}
\newcommand{\ind}{\mathrm{ind}}
\newcommand{\Iso}{\mathrm{Iso}\,}
\newcommand{\uIso}{\underline{\mathrm{Iso}}}
\newcommand{\uK}{\underline{K}}

\newcommand{\Lie}{\mathrm{Lie}\,}
\newcommand{\uLoc}{\mathbf{Loc}}

\newcommand{\rank}{\mathrm{rank}}
\newcommand{\Rep}{\mathbf{Rep}\,}
\newcommand{\uRep}{\mathbf{Rep}}
\newcommand{\uIRep}{\mathbf{IRep}}
\newcommand{\sep}{\mathrm{sep}}
\newcommand{\Sym}{\mathrm{Sym}}

\newcommand{\uVec}{\mathbf{Vec}}

\newcommand{\Ah}{{\mathcal A}}
\newcommand{\Eh}{{\mathcal E}}

\newcommand{\Gh}{{\mathcal G}}

\newcommand{\Mh}{{\mathcal M}}

\newcommand{\Oh}{{\mathcal O}}

\newcommand{\Th}{{\mathcal T}}
\newcommand{\emm}{{\mathfrak{m}}}

\newcommand{\eU}{\mathfrak{U}}

\newcommand{\ol}{\overline{l}}

\newcommand{\oK}{\overline{K}}
\newcommand{\oG}{\overline{G}}

\newcommand{\oU}{\overline{U}}
\newcommand{\tG}{\tilde{G}}
\newcommand{\tg}{\tilde{g}}

\newcommand{\tpi}{\tilde{\pi}}

\newcommand{\tY}{\tilde{Y}}

\newcommand{\silo}{\stackrel{\sim}{\rightarrow}}
\newtheorem{theorem}{Theorem}[section]

\newtheorem{prop}[theorem]{Proposition}

\newtheorem{cor}[theorem]{Corollary}

\newtheorem{remark}[theorem]{Remark}

\newcounter{aufz}
\newenvironment{rem}{\noindent {\bf Remark}}{}
\newenvironment{rems}{\noindent {\bf Remarks}}{}

\newenvironment{exmp}{\noindent{\bf Example}}{}

\newenvironment{proof}{\noindent {\bf Proof}}{\mbox{}\hspace*{\fill}$\Box$}

\newcommand{\tX}{\tilde{X}}
\newcommand{\Fh}{\mathcal{F}}

\newcommand{\verk}{\mbox{\scriptsize $\,\circ\,$}}

\begin{document}
\title{A pro-algebraic fundamental group for topological spaces\\
{\large \it Dedicated to Alexei Nikolaevich Parshin with admiration}}
\author{Christopher Deninger\footnote{Funded by the Deutsche Forschungsgemeinschaft (DFG, German Research Foundation) under Germany's Excellence Strategy EXC 2044--390685587, Mathematics M\"unster: Dynamics--Geometry--Structure.}}
\date{}
\maketitle




\section{Introduction}
In \cite{KS} Kucharczyk and Scholze define the ``\'etale fundamental group'' $\pi^{\et}_1 (X,x)$ of a pointed connected topological space $(X,x)$ using the method of Galois categories in \cite{SGA1}. This is a pro-finite group which classifies covering spaces of $X$ with finite fibres. For path-connected, locally path-connected and semi-locally simply connected spaces, $\pi^{\et}_1 (X,x)$ is the pro-finite completion of the ordinary fundamental group of $X$ i.e.
\begin{equation}\label{eq:1}
\pi^{\et}_1 (X,x) = \widehat{\pi_1 (X,x)} \; .
\end{equation}
For more general spaces, $\pi_1 (X,x)$ carries natural topologies and the relation of $\pi^{\et}_1 (X,x)$ to the (quasi-)topological group $\pi_1 (X,x)$ is also studied in \cite{KS}. Kucharczyk and Scholze make the following use of their \'etale fundamental group. For any field $F$ of characteristic zero containing all roots of unity, they construct a functorial compact connected Hausdorff space $X_F$ whose \'etale fundamental group $\pi^{\et}_1 (X_F ,x)$ is isomorphic to the absolute Galois group $G_F$ of $F$. The image of the usual fundamental group $\pi_1 (X_F , x)$ in $\pi_1^{\et} (X_F, x) \cong G_F$ is then an interesting extra structure of $G_F$.

For the well behaved topological spaces $X$ in \eqref{eq:1} the representations of $\pi_1 (X,x)$ correspond to local systems. This is not at all true in general. In the present note, we therefore study another type of fundamental group for pointed connected topological spaces $(X,x)$. Given a ground field $K$ it is an affine group scheme over $K$ which classifies the local systems of finite dimensional $K$-vector spaces on $X$. More precisely, the $\otimes$-category of such local systems together with the fibre functor $\omega_x$ in $x$ forms a neutralized Tannakian category over $K$. Define $\pi_K (X,x)$ to be its Tannakian dual i.e. 
\[
\pi_K (X,x) = \uAut^{\otimes} (\omega_x) \; .
\]
The group scheme $\pi_0 (\pi_K (X,x))$ of connected components of $\pi_K (X,x)$ is the maximal pro-\'etale quotient $\pi_K (X,x)^{\et}$ of $\pi_K (X,x)$, c.f. \cite[6.7]{W}. In section \ref{sec:3} we show that it is canonically isomorphic to $\pi^{\et}_1 (X,x)$ viewed as a pro-\'etale group scheme over $K$,
\begin{equation}
\label{eq:2}
\pi_K (X,x)^{\et} = \pi^{\et}_1 (X,x) /_K \; .
\end{equation}

For path-connected, locally path-connected and semi-locally simply connected spaces, $\pi_K (X,x)$ is isomorphic to the pro-algebraic completion over $K$ of the ordinary fundamental group
\begin{equation}
\label{eq:3}
\pi_K (X,x) = \pi_1 (X,x)^{\alg} \; .
\end{equation}
Recall that the last condition means that every point of $X$ has a neighborhood $U$ such that every loop in $U$ is nullhomotopic in $X$. See \cite[\S\,2]{BL} for the pro-algebraic completion of a group.

There are very interesting structural results about $\pi_K (X,x)$ for K\"ahler manifolds in \cite{P}, \cite{S}, Section 6. In the present note we focus on those properties of $\pi_K (X,x)$ that hold for very general, not even locally connected spaces $X$.

Section \ref{sec:2} is devoted to structural results about $\pi_K (X,x)$ and its algebraic quotients -- the monodromy group schemes of local systems. We also construct a certain pseudo\nobreakdash-torsor $P_X$ for the pro-discrete group $\pi_K (X,x) (K)$ which can serve as a replacement for the universal covering of $X$. Pullback to $P_X$ trivializes all local systems of finite dimensional $K$-vector spaces on $X$. 

It turns out that $\pi_K (X,x)$ is the projective limit of the Zariski closures of discrete subgroups of $\GL_r (K)$ for $r \ge 1$. In particular $\pi_K (X,x)$ is reduced even if the characteristic of $K$ is positive. This is easy to show if the connected topological space $X$ is also locally connected, mainly because for such spaces the connected components of a covering are again coverings. For the proof in the general case, we adapt a basic construction in algebraic geometry due to Nori \cite{N} to our situation.

Kucharczyk and Scholze show that $\pi^{\et}_1 (X,x)$ commutes with projective limits of connected compact Hausdorff spaces. In section \ref{sec:4} we prove a corresponding result for $\pi_K (X,x)$. This allows to calculate $\pi_K (X,x)$ for certain solenoidal spaces. These examples show that $\pi_K (X,x)$ can be non-trivial in cases where the \v{C}ech fundamental group is trivial. We also give some relations of $\pi_K (X,x)$ to cohomology, and end with some open questions.

Our ultimate motivation for introducing $\pi_K (X,x)$ is the hope to relate the motivic Galois group $G_{\Mh_F}$ over $\C$ to the pro-algebraic group $\pi_{\C} (X,x)$ of a suitable topological space $X$, generalizing the basic idea of \cite{KS}. This will not work with the space $X_F$ of \cite{KS} because as pointed out in the introduction of \cite{KS} the Steinberg relations do not hold in the rational cohomology of $X_F$. 

Generalizing Grothendieck's pro-finite fundamental group of a pointed topos which classifies finite coverings, Kennison has introduced an internal fundamental group of an (unpointed) topos using torsors which generalizes the functor $\Hom (\pi_1 (X) , -)$ \cite{K}. He does not use the Tannakian formalism but the respective fundamental groups could be related.  

I am grateful to the Newton Institute where part of this note was written and to Peter Scholze for explaining an argument in \cite{KS}. I would like to thank the referee for helpful comments, in particular improving Proposition \ref{t22} c).

\section{The pro-algebraic fundamental group} \label{sec:2}
Given a field $K$ with the discrete topology, a flat $K$-vector bundle on a topological space $X$ is a continuous map $\pi : E \to X$ whose fibres are finite dimensional $K$-vector spaces and such that locally on $X$ the map $\pi$ is isomorphic to the projection $X \times K^r \to X$ for some $r \ge 0$. Note that the transition functions between local trivializations are locally constant since $K$ has the discrete topology. The suffix ``flat'' is not really necessary. However, for fields $K$ like $\R$ or $\C$ which usually carry a different topology, ``flat'' serves as a reminder to view $K$ with the discrete topology. We write $\Gamma (X,E)$ for the $K$-vector space of continuous sections of $E$. Let $\uFl (X) = \uFl_K (X)$ denote the category of flat $K$-vector bundles with the obvious morphisms. We will often use the fact that a locally constant map from a connected topological space to a set is constant. The sheaf of continuous sections of a flat $K$-vector bundle is a local system $\Eh$ of finite-dimensional $K$-vector spaces on $X$ i.e. a sheaf of $K$-vector spaces on $X$ which is locally isomorphic to the constant sheaf $\uK^r$ for some integer $r \ge 0$. Let $\uLoc (X) = \uLoc_K (X)$ be the category of local systems of finite dimensional $K$-vector spaces. The functor from $\uFl (X)$ to $\uLoc (X)$ sending $E$ to $\Eh$ and correspondingly on morphisms is a equivalence of categories. A quasi-inverse is given by sending $\Eh$ to its espace \'etal\'e $E$ over $X$, c.f. \cite{B} Ch. I, 1.5. If $X$ is connected both categories are abelian. Note here that for a morphism of flat $K$-vector bundles $\varphi : E \to E'$ over $X$ the rank of the kernel $\ker \varphi_x$ is locally constant as a function of $x \in X$ and hence constant. If a group $G$ acts on a topological space $Y$ by homeomorphisms, we write $\uLoc^G_K (Y)$ for the category of $G$-locally constant sheaves of finite dimensional $K$-vector spaces. It is equivalent to the category $\uFl^G_K (Y)$ of flat finite rank $K$-vector bundles $E$ with a continuous $G$-action on the total space over the $G$-action on $Y$.

For the definition of a rigid abelian tensor category $\Th$ over $K$ we refer to \cite{DM} 2.1. It is an abelian category with a functor $\otimes : \Th \times \Th \to \Th$, a unit object $1$ and a dual object $E^{\vee}$ for any object $E$ together with morphisms $\ev : E \otimes E^{\vee} \to 1$ and $\delta : 1 \to E \otimes E^{\vee}$. All these data have to satisfy several compatibility conditions and all objects are reflexive $E \silo E^{\vee\vee}$. Moreover an isomorphism $K \silo \End (1)$ is part of the structure. For a connected topological space $X$, the categories $\uFl_K (X)$ and $\uLoc_K (X)$ are naturally rigid abelian tensor categories over $K$, and $E \mapsto \Eh$ is a tensor equivalence. The unit object is given by the trivial line bundle $\uK$ resp. the constant sheaf $\uK$. 

A neutral Tannakian category over $K$ is a rigid abelian tensor category $\Th$ over $K$ which admits a faithful $K$-linear exact $\otimes$-functor $\omega$ into the category $\uVec_K$ of finite dimensional vector spaces over $K$. Given such a ``fibre-functor'' $\omega$, the $\otimes$-functor $\Th$ is tensor-equivalent to the tensor category $\uRep_K (G)$ of finite dimensional $K$-representations of the affine $K$-group scheme $G = \uAut^{\otimes} (\omega)$, the Tannakian dual of $(\Th , \omega)$. Under the equivalence the fibre functor $\omega$ on $\Th$ becomes the forgetful functor on $\uRep_K (G)$. Here, by definition we have for any $K$-algebra $R$
\[
G (R) = \uAut^{\otimes} (\omega) (R) = \Aut^{\otimes} (\phi_R \verk \omega) \; ,
\]
where $\phi_R = R \otimes_- : \uVec_K \to \uMod_R$. See \cite{DM} Theorem 2.11 for more details. For an object $E$ of $\Th$ the full $\otimes$-subcategory $\langle E \rangle^{\otimes}$ generated by $E$ is defined as the full subcategory of $\Th$ of objects isomorphic to a subquotient of $Q (E , E^{\vee})$ for some $Q \in \N [t,s]$ where $\N = \{ 0 , 1 , 2 , \ldots \}$. For a point $x$ of a topological space $X$ we have the usual fibre functor
\[
\omega_x : \uFl_K (X) \longrightarrow \uVec_K
\]
sending $E$ to $E_x$ and $\varphi : E \to E'$ to $\varphi_x : E_x \to E'_x$. It is a $K$-linear exact tensor functor.

\begin{prop} 
\label{t21}
For any pointed connected topological space $(X,x)$ and any field $K$, the category $\uFl_K (X)$ is neutral Tannakian with fibre functor $\omega_x$.
\end{prop}

\begin{proof}
We have to show that $\omega_x$ is faithful. Let $\varphi : E \to E'$ be a morphism in $\uFl_K (X)$. The function $X \to \Z$ sending $y \in X$ to the rank of $\varphi_y$ is locally constant, hence constant on $X$. If $\omega_x (\varphi) = 0$ i.e. $\rk \varphi_x = 0$ we therefore have $\rk \varphi_y = 0$ for all $y \in X$ and hence $\varphi = 0$. 
\end{proof}

For a connected space $X$ and a point $x \in X$, we denote the Tannakian dual of $(\uFl_K (X) , \omega_x)$ by $\pi_K (X,x)$. It is an affine group scheme over $K$ whose finite dimensional $K$-representations classify the flat bundles on $X$. 

\begin{cor}
\label{t23}
Let $X$ be a connected topological space and let $E$ be a trivial bundle in $\uFl_K (X)$. Then any subquotient $F$ of $E$ in $\uFl_K (X)$ is a trivial flat bundle.
\end{cor}

\begin{proof}
Fix a point $x \in X$. The bundle $E$ corresponds to a trivial representation $\rho_E$ of $\pi_K (X,x)$ on $E_x$. All subquotients of $\rho_E$ are trivial as well and hence $F$ is a trivial flat vector bundle. 
\end{proof}

For $E$ in $\uFl_K (X)$ the Tannakian dual of $(\langle E \rangle^{\otimes} , \omega_x)$ is a closed (algebraic) subgroup scheme of $\GL_{E_x}$ over $K$, the {\it monodromy group scheme} of $E$,
\[
G_E = G_{E,x} = \uAut^{\otimes} (\omega_x \mid \langle E \rangle^{\otimes}) \subset \GL_{E_x} \;.
\]
The induced morphism $\pi_K (X,x) \to G_E$ is faithfully flat by \cite{DM}, Proposition 2.21. Hence $G_E$ is the image of the representation $\pi_K (X,x) \to \GL_{E_x}$ corresponding to $E$.

We now fix some notations. Let $G$ be a topological group and $P$ a topological space on which $G$ acts continuously from the right. Let $G$ act trivially on a topological space $X$ and let $\pi : P \to X$ be a continuous $G$-equivariant map. We call $P$ a {\it (surjective) pseudo $G$-torsor} over $X$ if the continuous map
\[
P \times G \longrightarrow P \times_X P \, , \, (p , \sigma) \mapsto (p , p^{\sigma})
\]
is a homeomorphism (and $\pi$ is surjective). If in addition $\pi$ has continuous sections locally, then $P$ is a {\it $G$-torsor}, or {\it principal homogenous $G$-bundle}. Equivalently, $P$ is a trivial $G$-torsor locally. If $G$ has the discrete topology then any $G$-torsor is in particular a covering. Connected $G$-torsors for a discrete group $G$ are the same as Galois coverings with group $G$. 

We need the following fact which follows from \cite{G} 5.3 by restricting to the subcategories of locally constant sheaves and noting that the inverse image of sheaves commutes with tensor products. Let $G$ be a discrete group and let $\pi : P  \to X$ be a $G$-torsor. In particular, $\pi$ is a covering, hence a local homeomorphism and $G$ acts without fixed points and properly discontinously such that $P / G \silo X$. There is an equivalence of $\otimes$-categories between $\uLoc_K (X)$ and $\uLoc^G_K (P)$. The functors
\begin{equation}
\label{eq:4}
\uLoc_K (X) \overset{\displaystyle\xrightarrow{\;\pi^{-1}}}{\xleftarrow[\; \pi^G_*\;]{}} \uLoc^G_K (P)
\end{equation}
are quasi-inverses of each other. Here $\pi^G_*$ is defined by setting
\[
(\pi^G_* \Fh) (U) = \Fh (\pi^{-1} (U))^G \quad \text{for} \; U \subset X \; \text{open} \; .
\]
There are corresponding quasi-inverse functors $\pi^*$ and $\pi^G_*$ between $\uFl_K (X)$ and $\uFl^G_K (P)$. For a representation of $G$ on a finite dimensional $K$-vector space $V$, the bundle $P \times V$ with the diagonal $G$-action is in $\uFl^G_K (P)$ and we have
\[
\pi^G_* (P \times V) = P \times^G V := (P \times V) / G \; .
\]
Let $G$ be a subgroup of $\GL_r (K)$ with the discrete topology and assume that the bundle $E$ in $\uFl_K (X)$ has a reduction of structure group to $G$. This means that there is a $G$-torsor $\pi : P \to X$ such that $E \cong P \times^G K^r$. Let $ \oG$ be the Zariski closure of $G$ in the group scheme $\GL_r$ over $K$. In this situation we have the following information:

\begin{prop}
\label{t22}
a) For a suitable isomorphism $E_x = K^r$ of $K$-vector spaces, $G_{E,x}$ is a closed subgroup scheme of $\oG$ in $\GL_{E_x} = \GL_r$. If $P$ is connected we even have $G_{E,x} = \oG$.\\
b) If the connected topological space $X$ is also locally connected, then the structure group of $E$ can be reduced to a subgroup $G$ of $\GL_r (K)$ such that $P$ is connected, and hence $G_{E,x} = \oG$. \\
c) In the situation of b), the affine group scheme $\pi_K (X,x)$ is a projective limit of algebraic groups of the form $\oG$ where the $G$'s are discrete subgroups of $\GL_r (K)$ for varying $r$. In particular $\pi_K (X,x)$ is geometrically reduced.
\end{prop}

\begin{proof}
a) We are given a bundle $E$ in $\uFl_K (X)$ and a $G$-torsor $\pi : P \to X$ for a subgroup $G \subset \GL_r (K)$ such that $E \cong P \times^G K^r$ and hence $\pi^* E \cong P \times K^r$ in $\uFl^G_K (P)$. Hence we have an equivalence of tensor categories
\[
\langle E \rangle^{\otimes} \cong \langle P \times K^r  \rangle^{\otimes} \; . 
\]
Here on the right we mean the full $\otimes$-subcategory of $\uFl^G_K (P)$ generated by the $G$-bundle $P \times K^r$ over $P$. Choosing a point $p \in P$ with $\pi (p) = x$ we get an isomorphism of affine group schemes over $K$
\[
G_{E,x} = \uAut^{\otimes} (\omega_x \mid \langle E \rangle^{\otimes}) \cong \uAut^{\otimes} (\omega_p \mid \langle P \times K^r \rangle)^{\otimes}) =: \tG \; .
\]
The point $p$ and the isomorphism $E \cong P  \times^G K^r$ determine an isomorphism
\[
E_x \cong \pi^{-1} (x) \times^G K^r \cong \{ p \} \times K^r \cong K^r \; .
\]
Viewing this isomorphism as an identification, the isomorphism $G_{E,x } \cong \tG$ becomes an equality $G_{E,x} = \tG$ of closed subgroup schemes of $\GL_r$. 

Let $\rho : G  \hookrightarrow \GL_r (K)$ denote the inclusion and consider the Tannakian subcategory $\langle \rho \rangle^{\otimes}$ of $\uRep_K (G)$ generated by $\rho$. It follows from the proof of \cite{DM} Proposition 2.8 that we have
\[
\oG = \uAut^{\otimes} (\omega \mid \langle \rho\rangle^{\otimes}) \; .
\]
Here $\omega$ is the fibre functor forgetting the $G$-module structure. The natural $\otimes$-functor
\[
\alpha : \uRep_K (G) \longrightarrow \uFl^G_K (P) \; , \; V \longmapsto P \times V \quad \text{(diagonal $G$-action)}
\]
sends $\rho$ to $P \times K^r$ and therefore restricts to a $\otimes$-functor
\[
\alpha_{\rho} : \langle \rho \rangle^{\otimes} \longrightarrow \langle P \times K^r\rangle^{\otimes} \; .
\]
The functor $\alpha$ is faithful and if $P$ is connected even fully faithful. Namely for representations $\rho_i : G \to \GL (V_i)$ in $\uRep_K (G)$ for $i = 1,2$, a morphism
\[
\varphi : \alpha (\rho_1) = P \times V_1 \longrightarrow P \times V_2 = \alpha (\rho_2)
\]
is a continuous map of the form $\varphi (p , v) = (p , \varphi_p (v))$ for $p \in P , v \in V_1$. Here $\varphi_p \in \Hom_K (V_1 , V_2)$ for each $p \in P$. The continuous map $p \mapsto \varphi_p$ is locally constant hence constant on the connected components of $P$, since $\Hom_K (V_1, V_2)$ carries the discrete topology. Moreover we have
\[
\rho_2 (\sigma) (\varphi_p (v)) = \varphi_{p^{\sigma^{-1}}} (\rho_1 (\sigma) v) \quad \text{for} \; v \in V_1 , p \in P , \sigma \in G \; .
\]
The morphisms $\varphi$ in the image of $\alpha$ are those where $\varphi_p$ is independent of $p$.
If $P$ is connected, it follows that $\alpha$ and hence $\alpha_{\rho}$ are fully faithful. Applying Corollary \ref{t23} to $X = P$ it also follows that $\alpha_{\rho}$ is essentially surjective, and hence an equivalence of $\otimes$-categories. It follows that $\oG = \tG$ as closed subgroup schemes of $\GL_r$ over $K$, and hence $\oG = G_{E , x}$. If $P$ is not necessarily connected we only obtain a morphism $\tG \to \oG$ of affine group schemes, which is compatible with the closed immersions of $\tG$ and $\oG$ into $\GL_r$. Hence $G_{E,x} = \tG \to \oG$ is a closed immersion as well.\\
b) Choose a connected component $P_0$ of $P$ and let $G_0$ be the stabilizer of $P_0$ for the $G$-action on the set of connected components of $P$. Then the map
\begin{equation}
\label{eq:5}
P_0 \times G_0 \silo P_0 \times_X P_0 \; , \; ( p , \sigma) \mapsto (p , p^{\sigma})
\end{equation}
is a homeomorphism. This follows from the corresponding property of the $G$-action on $P$, because $p' = p^{\sigma}$ for $p , p' \in P_0$ and $\sigma \in G$ implies that $\sigma \in G_0$. The component $P_0$ is open in $P$. Namely, for $p_0 \in P_0$ choose a connected open neighborhood $x_0 \in U \subset X$ where $x_0 = \pi (p_0)$ such that $P \, |_U := \pi^{-1} (U) \to U$ is $G$-equivariantly homeomorphic to $U \times G$ over $U$. Viewing this as an identification we have $p_0 = (x_0 , g)$ for some $g \in G$. Since $U \times \{ g \}$ is connected and $P_0$ is the connected component of $p_0$, it follows that $U \times \{ g \} \subset P_0$. But $U \times \{ g \}$ is open in $P$ since $G$ carries the discrete topology and hence $P_0$ is open. By \eqref{eq:5} the action of $G_0$ on the fibres of $P_0 \to X$ is simply transitive and hence the inclusion $U \times \{ g \} \subset P_0$ implies that there is a $G_0$-equivariant homeomorphism $P_0 \, |_U = U \times g G_0$. Thus $P_0$ is a $G_0$-torsor if $\pi (P_0) = X$. Since $\pi$ and $P_0$ are open $\pi (P_0)$ is open. We show that $\pi (P_0)$ is closed. Assume $y \in X \setminus \pi (P_0)$. Choose an open connected neighborhood $y \in U \subset X$ such that $P \, |_U = U \times G$ as above. If $\pi (P_0) \cap U \neq \emptyset$, then $(x,g) \in P_0$ for some $x \in U , g \in G$, and hence $U \times  \{ g \} \subset P_0$ which implies $y \in \pi (P_0)$ a contradiction. Thus $\pi (P_0) \cap U = \emptyset$ and therefore $X \setminus \pi (P_0)$ is open. Since $\pi (P_0) \neq \emptyset$ is open and closed and $X$ is connected, it follows that $\pi (P_0) = X$. Since $(\pi \, |_{P_0})^* E$ is trivial the structure group of $E$ can be reduced to $G_0$ and using a) we have $G_{E,x} = \oG_0$ after an appropriate identification $E_x = K^r$. \\
The first assertion in c) is a formal consequence of b) since $\uFl_K (X)$ is the filtered inductive limit of the Tannakian subcategories $\langle E \rangle^{\otimes}$. Note that given $E_1 , E_2$, we have $\langle E_i\rangle^{\otimes} \subset \langle E\rangle^{\otimes}$ where $E = E_1 \oplus E_2$. For the algebraic group $\oG = \spec A$ the $K$-algebra $A$ has no non-zero nilpotent elements. Since $G \subset \oG (K)$ is dense in $\oG$ by definition, it follows that for every $0 \neq a \in A$ there is a map of $K$-algebras $\chi_g : A \to K$ corresponding to an element $g \in G$ such that $\chi_g (a) = a (g) \neq 0$. It follows that we have an injective map of $K$-algebras $A \to K^G , a \mapsto (\chi_g (a))_g$. Tensoring with a finite extension $L / K$ we obtain $L$-algebra injections
\[
A \otimes_K L \hookrightarrow K^G \otimes_K L \hookrightarrow L^G \; .
\]
Hence $A \otimes_K L$ has no non-zero nilpotent elements. Thus $\oG = \spec A$ and therefore also $\pi_K (X ,x)$ are geometrically reduced over $K$. 
\end{proof}

\begin{rem}
Over non-locally connected topological spaces $X$ the connected components $P_0$ of a $G$-torsor $P$ need not map surjectively to $X$. Consider the connected solenoid
\[
X = \varprojlim (\ldots  \xrightarrow{p} \R / \Z \xrightarrow{p} \R / \Z) = \R \times^{\Z} \Z_p \; .
\]
The covering $P = \R \times \Z_p \to X$ is a $\Z$-torsor with connected components $P_a = \R \times \{ a \} $ for $a \in \Z_p$ all of which have trivial stabilizer groups in $\Z$. The images of the $P_a$'s are the path-connected components of $X$, of which there are uncountably many. No $P_a$ maps surjectively to $X$. I learned about this example in a post by Taras Banakh on MathOverflow.
\end{rem}

Using ideas of Nori from another context \cite{N}, we will now construct pseudo-torsors trivializing flat bundles. In particular, for arbitrary connected topological spaces and algebraically closed $K$, we get a surjective pseudo-torsor for a pro-discrete topological group which can serve as a replacement for the universal covering space, the latter existing only for well behaved spaces. In the examples that we are aware of, our pseudo-torsors are actually torsors. We do not know if this is true in general.

The construction goes as follows. Let $(X , x)$ be a pointed connected topological space and $E$ a flat bundle in $\uFl_K (X)$. Consider a composition of morphisms of group schemes over the field $K$
\begin{equation}
\label{eq:6n}
\pi_K (X,x) \overset{\alpha}{\twoheadrightarrow} G_K \overset{\beta}{\twoheadrightarrow} G_E \subset \GL_{E_x} \; ,
\end{equation}
where $\alpha$ and $\beta$ are faithfully flat. The cases $G_K = \pi_K (X,x)$ and $G_K = G_E$ are especially relevant. Consider the full embeddings and the equivalence of $\otimes$-categories
\[
\uRep_K (G_E) \overset{\beta^*}{\hookrightarrow} \uRep_K (G_K) \overset{\alpha^*}{\hookrightarrow} \uRep_K (\pi_K (X,x)) \simeq \uFl_K (X) \; .
\]
They extend to the $\ind$-categories, which are again abelian tensor categories
\begin{equation}
\label{eq:7n}
\uIRep_K (G_E) \overset{\beta^*}{\hookrightarrow} \uIRep_K (G_K) \overset{\alpha^*}{\hookrightarrow} \uIRep_K (\pi_K (X,x)) \simeq \uIFl_K (X) \; .
\end{equation}
Let $H$ be a commutative Hopf algebra over a field. By \cite{W} 3.3 Theorem, every comodule $V$ for $H$ is the directed union of finite-dimensional subcomodules. It follows that for any affine group scheme $\pi_K$ over $K$, the category $\uIRep_K (\pi_K)$ can be identified with the category of comodules for $H = \Gamma (\pi_K , \Oh)$. 

Let $G^{\delta}_K$ be the group scheme $G_K$ together with the trivial (left-)action of $G_K$. Then the multiplication morphism
\begin{equation}
\label{eq:8n}
m : G_K \times G^{\delta}_K \longrightarrow G_K \; , \; (\sigma , \tau) \longmapsto \sigma \tau \quad \text{(on $S$-valued points)}
\end{equation}
is $G_K$-equivariant under the natural left $G_K$-actions. Similarly, let $\A^{\delta}_{E_x}$ be the affine space over $E_x$ with the trivial (left-) action by $G_K$. The left action of $G_K$ via $\beta$ on $\A_{E_x}$ gives a $G_K$-equivariant morphism
\begin{equation}
\label{eq:9n} 
m_E : G_K \times \A^{\delta}_{E_x} \longrightarrow \A_{E_x} \; , \; (\sigma , v) \longmapsto \sigma v \quad \text{(on $S$-valued points).}
\end{equation}
The resulting $G_K$-equivariant morphisms
\begin{equation}
\label{eq:10n}
G_K \times G^{\delta}_K \silo G_K \times G_K \; , \; (\sigma , \tau) \mapsto (\sigma , \sigma \tau)
\end{equation}
and
\begin{equation}
\label{eq:11n}
G_K \times \A^{\delta}_{E_x} \silo G_K \times \A_{E_x} \; , \; (\sigma , v) \mapsto (\sigma , \sigma v)
\end{equation}
are isomorphisms. Passing to global sections with the induced left $G_K$-action (or $\Gamma (G_K)$-comodule structure), we obtain morphisms of commutative unital $K$-algebra objects in $\uIRep_K (G_K) \overset{\alpha^*}{\hookrightarrow} \uIRep_K (\pi_K (X,x))$
\begin{equation}
\label{eq:12n} 
\Delta = m^* : \Gamma (G_K) \longrightarrow \Gamma (G_K) \otimes \Gamma (G_K)^{\delta}
\end{equation}
and
\begin{equation}
\label{eq:13n}
\Delta_E = m^*_E : \Gamma (\A_{E_x}) \longrightarrow \Gamma (G_K) \otimes \Gamma (\A_{E_x})^{\delta} \; ,
\end{equation}
and isomorphisms
\begin{equation}
\label{eq:14n}
\Gamma (G_K) \otimes \Gamma (G_K) \silo \Gamma (G_K) \otimes \Gamma (G_K)^{\delta} \; , \; a \otimes b \mapsto (a \otimes 1) \Delta (b)
\end{equation}
and
\begin{equation}
\label{eq:15n}
\Gamma (G_K) \otimes \Gamma (\A_{E_x}) \silo \Gamma (G_K) \otimes \Gamma (\A_{E_x})^{\delta} \; , \; a \otimes v \mapsto (a \otimes 1) \Delta_E (v) \; .
\end{equation}

Under the tensor equivalence of $\uIRep_K (\pi_K (X,x))$ with $\uIFl_K (X)$ let $A = A (G_K) = A (\alpha)$ and $A_E$ be the commutative $K$-algebra objects in $\uIFl_K (X)$ corresponding to $\Gamma (G_K)$ resp. $\Gamma (\A_{E_x})$. The objects corresponding to $\Gamma (G_K)^{\delta}$ and $\Gamma (\A_{E_x})^{\delta}$ are the trivial bundles of $K$-algebras
\[
\underline{\Gamma (G_K)} = X \times \Gamma (G_K) \quad \text{and} \quad \underline{\Gamma (\A_{E_x})} = X \times \Gamma (\A_{E_x}) \; .
\]
We get morphisms of unital $K$-algebra objects in $\uIFl_K (X)$,
\begin{equation}
\label{eq:16n}
\Delta : A \longrightarrow A \otimes \underline{\Gamma (G_K)}
\end{equation}
and
\begin{equation}
\label{eq:17n}
\Delta_E : A_E \longrightarrow  A\otimes \underline{\Gamma (\A_{E_x})}
\end{equation}
and isomorphisms
\begin{equation}
\label{eq:18n}
A \otimes A \silo A \otimes \underline{\Gamma (G_K)} \; , \; a \otimes b \longmapsto (a \otimes 1) \Delta (b)
\end{equation}
and 
\begin{equation}
\label{eq:19n}
A \otimes A_E \silo A \otimes \underline{\Gamma (\A_{E_x})} \; , \; a \otimes v \longmapsto (a \otimes 1) \Delta (v) \; .
\end{equation}
For an object $B = \varinjlim_i B_i$ in $\uIFl_K (X)$, filtered inductive limit of flat bundles $B_i$ in $\uFl_K (X)$, let
\[
\uHom (B, \uK) = \varprojlim_i \uHom (B_i , \uK) = \coprod_{x \in X} \varprojlim_i \Hom (B_{ix} , K)
\]
be the projective limit of the total spaces of the $\Hom$-bundles $\uHom (B_i , \uK)$. Then $\uHom (B, \uK)$ is a topological space with a continuous surjective map $\pi$ to $X$. If $B$ is a $K$-algebra object in $\uIFl_K (X)$, let $\uHom_{\alg} (B,\uK)$ be the subspace of $\uHom (B,K)$ consisting of fibrewise algebra homomorphisms. Thus its fibre over the point $z \in X$ is
\[
\uHom (B, \uK)_z = \Hom_{\alg} (B_z , K) \; ,
\]
where $B_z$ is the $K$-algebra $B_z = \varinjlim_i B_{iz}$. Already for $X$ a point we see that $B_z$ may be empty. For $B = \underline{\Gamma (G_K)}$ we have
\[
\uHom_{\alg} (\underline{\Gamma (G_K)} , \uK) = X \times \Hom_{\alg} (\Gamma (G_K) , K) = X \times G_K (K) = \underline{G_K (K)} \; .
\]
Here $G_K (K)$ carries the pro-discrete topology. For $B = A$ set
\begin{equation}
\label{eq:20n}
P = P (G_K) = P (\pi_K (X,x) \to G_K) := \uHom_{\alg} (A , \uK) \; .
\end{equation}
Its fibre over the fixed base point $x \in X$ is
\begin{equation}
\label{eq:21n}
P_x = \Hom_{\alg} (\Gamma (G_K) , K) = G_K (K) \; .
\end{equation}

For path-connected spaces using the map \eqref{eq:7} below it is easy to see that $P_z \neq \emptyset$ for all $z \in Z$. If $X$ is connected and $G_K$ is algebraic, we still have $P_z \neq \emptyset$ for $z \in X$, c.f. Theorem \ref{t25n} below. For arbitrary quotients $G_K$, if $K$ is algebraically closed we can use a result of Deligne in \cite{D}, see also \cite{Cou}, \cite{Wib}. He proved that for a Tannakian category over an algebraically closed field $K$, any two fibre functors over $K$ are $\otimes$-isomorphic. In particular there is a $\otimes$-isomorphism $\gamma : \omega_x \silo \omega_z$ on $\uFl_K (X)$. It extends naturally to $\uIFl_K (X)$ where $\omega_x , \omega_z$ now take values in arbitrary $K$-vector spaces. It follows that $\gamma (A)$ gives an isomorphism of $K$-algebras (!) $A_x \silo A_z$, and therefore a homeomorphism of pro-discrete spaces
\[
G_K (K) = P_x \silo P_z \; .
\]
In particular, for algebraically closed $K$, the continuous map $\pi : P \to X$ is always surjective. For any field $K$, applying the functor $\uHom_{\alg} (\_ ,\uK)$ to \eqref{eq:16n} we obtain a continuous map
\begin{equation}
\label{eq:22n}
P \times G_K (K) = P \times_X \underline{G_K (K)} \longrightarrow P
\end{equation}
which by construction is a continuous $G_K (K)$-operation on $P$. Similarly the isomorphism \eqref{eq:18n} gives a homeomorphism
\begin{equation}
\label{eq:23n}
P \times G_K (K) \silo P \times_X P \; .
\end{equation}
Thus $P$ is a pseudo-torsor for $G_K (K)$. It is surjective if $K = \oK$. There are (fibrewise $K$-linear) homeomorphisms
\[
\uHom_{\alg} (\underline{\Gamma (\A_{E_x}}) , \uK) = X \times \Hom_{\alg} (\Gamma (\A_{E_x}) , K) = X \times \A_{E_x} (K) = X \times E_x \, .
\]
In $\uIRep_K (\pi_K (X,x))$ we have
\[
\Gamma (\A_{E_x}) = \Sym (\check{E}_x) \; .
\]
Under the $\otimes$-equivalence with $\uIFl_K (X)$ we find
\[
A_E = \Sym \, \check{E} \; ,
\]
and therefore
\[
\uHom_{\alg} (A_E , \uK) = \uHom (\check{E} , \uK) = E \; .
\]
Applying $\uHom_{\alg} (\_ , \uK)$ to \eqref{eq:19n} we get an isomorphism of vector bundles
\begin{equation}
\label{eq:24n}
\pi^* E = P \times_X E \silo P \times E_x \; .
\end{equation}
Thus the flat vector bundle $E$ is trivialized by pullback to $P$. One checks that the natural right $G_K (K)$-action on $\pi^* E$ over the one on $P$ corresponds to the diagonal right action on $P \times E_x$ given by
\[
(p , v) \sigma = (p \sigma , \rho (\sigma^{-1}) v) \quad \text{for} \; p \in P , v \in E_x , \sigma \in G_K (K) \; .
\]
Here $\rho : G_K (K) \to \GL (E_x)$ is induced from the given morphism $G_K \to \GL_{E_x}$. 

For $G_K = \pi_K (X,x)$ consider the continuous map
\begin{equation}
\label{eq:25n}
\pi : P_X := P (\pi_K (X,x)) \longrightarrow X \; .
\end{equation} 
Then $P_X$ is a pseudo-torsor for the pro-discrete group $\pi_K (X,x) (K)$. If $K = \oK$ then $\pi$ is surjective. Pullback via $\pi$ trivializes every flat bundle in $\uFl_K (X)$. In this regard $P_X$ may be viewed as a replacement for the universal covering. See Theorem \ref{t41} for a description of $P_X$ in the classical case. For $G_K = G_E$, we set 
\begin{equation}
\label{eq:26n}
P_E = P (G_E) \xrightarrow{\pi} X \; .
\end{equation}
The pullback $\pi^* E$ is trivial by \eqref{eq:24n}. Thus we have proved the following result.

\begin{theorem}
\label{t24n}
Let $(X,x)$ be a pointed connected topological space and $K$ a field.\\
1) Given a faithfully flat quotient $\pi_K (X,x) \xrightarrow{\alpha} G_K$ over $K$, consider the space $P = P (G_K) \xrightarrow{\pi} X$ defined in \eqref{eq:20n}. Pullback along $\pi$ trivializes every bundle $E$ in $\uFl_K (X)$ canonically, whose monodromy representation $\pi_K (X,x) \to G_E \subset \GL_{E_x}$ factors over $\alpha$, c.f. \eqref{eq:24n}. In particular pullback to $P_X = P (\pi_K (X,x))$ trivializes every bundle in $\uFl_K (X)$ and pullback to $P_E = P (G_E)$ trivializes $E$.\\
2) The pro-discrete group $G_K (K)$ acts continuously and simply transitively on the non-empty fibres of $\pi : P \to X$ and $P$ is a pseudo-torsor. If $X$ is path connected or if $K$ is algebraically closed, then $\pi : P \to X$ is surjective. \\
\end{theorem}

We will now show that for any algebraic quotient $G_K$ of $\pi_K (X,x)$ the $G_K (K)$-pseudo-torsor $P = P (G_K)$ is a torsor. To do this we first recall that finitely generated $K$-algebras can be written as reflexive coequalizers of finitely generated polynomial rings over $K$. Namely for such a $K$-algebra $\Gamma$ choose a presentation
\[
0 \longrightarrow \Ah \longrightarrow K [x_1 , \ldots , x_n] \longrightarrow \Gamma \longrightarrow 0 \; .
\]
Then $\Ah$ is generated by finitely many polynomials $\Ah = \langle P_1 , \ldots , P_m \rangle$. Consider the diagram
\begin{equation}
\label{eq:27n}
\xymatrix{
K [x_1 , \ldots , x_n , t_1 , \ldots , t_m] &
\ar[l]_-{\overset{\xrightarrow{\; f \; }}{s}}^{\overset{\dis \longrightarrow}{g}} K [x_1 , \ldots , x_n] \longrightarrow \Gamma
}
\end{equation}
where the $K$-algebra homomorphisms $f,g$ and $s$ are defined as follows: $f (x_i) = g (x_i) = x_i$ and $s (x_i) = x_i$ for $1 \le i \le n$. Moreover $f (t_j) = P_j$ and $g (t_j) = 0$ for $1 \le j \le m$. Then we have $f \verk s = \id = g \verk s$ and a short calculation shows that the ideal $\imm (f-g) = \Ah$. We now use an invariant version of this construction to write $\Gamma (G_K)$ as a reflexive coequalizer in the category of commutative $K$-algebra objects in $\uIRep_K (\pi_K (X,x))$. Since $G_K$ is algebraic, the $K$-algebra $\Gamma (G_K)$ is finitely generated. Let $S$ be a finite set of generators. By \cite{W} 3.3 Theorem, there is a finite dimensional $\pi_K (X,x)$-subrepresentation $\cV_x \subset \Gamma (G_K)$ with $S \subset \cV_x$. Let $\cV$ be the corresponding bundle in $\uFl_K (X)$ and let $V$ be its dual. Define $\Ah$ by the exact sequence in $\uIRep_K (\pi_K (X,x))$
\[
0 \longrightarrow \Ah \longrightarrow \Sym \cV_x \longrightarrow \Gamma (G_K) \longrightarrow 0 \; .
\]
Again using \cite{W} 3.3 Theorem, we can choose a finite dimensional $\pi_K (X,x)$-subrepresentation $H \subset \Ah$ containing a finite set of generators of the finitely generated ideal $\Ah$. Then $\cW_x = \cV_x \oplus H$ is a finite dimensional representation of $\pi_K (X,x)$. Let $\cW$ be the corresponding bundle in $\uFl_K (X)$ and $W$ its dual. Consider the following diagram
\begin{equation}
\label{eq:28n}
\xymatrix{
\Sym \cW_x & \ar[l]_-{\overset{\xrightarrow{f_x}}{s_x}}^-{\overset{\dis \longrightarrow}{g_x}} \Sym \cV_x \longrightarrow \Gamma (G_K) \; .
}
\end{equation}
Here the $\pi_K (X,x)$-equivariant algebra homomorphisms $f_x , g_x , s_x$ are defined as follows:
\begin{align*}
f_x \, |_{\cV_x} & = g_x \, |_{\cV_x} = \id \quad \text{and} \quad s_x \, |_{\cV_x} = \, \text{inclusion}\, \cV_x \hookrightarrow \cV_x \oplus H = \cW_x\\
f_x \, |_H & = \text{inclusion} \, H \hookrightarrow \Sym \cV_x \; \text{and} \; g_x |_H = 0 \; .
\end{align*}
Since this is just an invariant way of writing \eqref{eq:27n} the diagram \eqref{eq:28n} exhibits $\Gamma (G_K)$ as a reflexive coequalizer in the category of commutative $K$-algebras. Since everything is $\pi_K (X,x)$-equivariant, \eqref{eq:28n} also describes $\Gamma (G_K)$ as a reflexive coequalizer in the category of commutative $K$-algebra objects of $\uIRep_K (\pi_K (X,x))$. 

Recall that $A$ was the commutative $K$-algebra object in $\uIFl_K (X)$ corresponding to $G_K$. Under the equivalence of Tannakian categories $\uIRep_K (\pi_K (X,x)) = \uIFl_K (X)$ we obtain a diagram in $\uIFl_K (X)$ whose fibre in $x$ is \eqref{eq:28n}
\begin{equation}
\label{eq:29n}
\xymatrix{\Sym \cW & \ar[l]^-{\overset{\dis \longrightarrow}{g}}_-{\overset{\xrightarrow{\; f \;}}{s}} \Sym \cV \longrightarrow A \; .
}
\end{equation}
The diagram \eqref{eq:29n} describes $A$ as a reflexive coequalizer in the category of commutative $K$-algebra objects in $\uIFl_K (X)$. Applying the functor $\uHom_{\alg} (\_ , \uK)$, recalling definition \eqref{eq:20n} and noting that
\[
\uHom_{\alg} (\Sym \cV , \uK) = \uHom (\cV , K) = V \; ,
\]
we get a diagram of spaces over $X$ making $P$ a reflexive equalizer
\begin{equation}
\label{eq:30n}
\xymatrix{
P \longrightarrow V & \ar[l]_-{\overset{\xrightarrow{\; F \;}}{S}}^-{\overset{\dis \longrightarrow}{G}}  W \quad \text{where} \; S \verk F = S \verk G = \id \; .
}
\end{equation}
Thus we get homeomorphisms of spaces over $X$
\[
P = \{ v \in V \mid F (v) = G (v) \} \overset{\diag}{\silo} \{ (v_1 , v_2) \in V \times V \mid F (v_1) = G (v_2) \} \; .
\]
For each point $z \in X$ the fibre $f_z$ of the map $f$ in \eqref{eq:29n} is a homomorphism of $K$-algebras
\[
f_z : \Sym \cW_z \longrightarrow \Sym \cV_z \; .
\]
Applying the functor $\spec$ we get a morphism of affine $K$-spaces over the fibres $V_z$ and $W_z$ of the bundles $V$ and $W$
\[
\spec f_z : \A_{V_z} \longrightarrow \A_{W_z} \; .
\]
The fibre $F_z$ of the map $F$ in \eqref{eq:30n} is obtained by passing to the $K$-valued points of $\spec f_z$
\[
F = (\spec f_z) (K) : \A_{V_z} (K) = V_z \longrightarrow W_z = \A_{W_z} (K) \; .
\]
For any choice of isomorphisms $V_z \cong K^n$ and $W_z \cong K^m$ the $m$ components of the map $F$ are given by polynomial functions in $n$ variables. Their maximal degree $d$ is the maximum of $1$ and the maximal degree of an element of $H_x$ in $\Sym \cV_x$. In particular $d$ is independent of $z \in X$. Choose an open neighborhood $z \in U \subset X$ such that there are trivializations $V \, |_U \cong \uK^n$ and $W \, |_U \cong \uK^m$. The induced continuous map
\[
F_U : U \times K^n \cong V \, |_U \xrightarrow{F \, |_U} W \, |_U \cong U \times K^m
\]
has the form
\[
F_U (y , \xi) = (y , \varphi_U (y , \xi)) \quad \text{for} \; y \in U \; , \; \xi \in K^n
\]
where
\[
\varphi_U (y , \xi)_j = \sum_{|\nu| \le d} a_{\nu , j} (y) \xi^{\nu} \quad \text{for} \; 1 \le j \le m \; .
\]
The coefficient functions $a_{\nu , j} : U \to K$ are continuous and hence locally constant. Since there are only finitely many of them, we may assume that they are constant by shrinking $U$. Thus we have
\[
F_U (y , \xi) = (y , \varphi (\xi)) \quad \text{for} \; y \in U \, , \, \xi \in K^n
\]
where
\[
\varphi (\xi)_j = \sum_{|\nu| \le d} a_{\nu , j} \xi^{\nu} \quad \text{with} \; a_{\nu , j} \in K \; \text{for} \; 1 \le j \le m \; .
\]
The map $G$ is a linear map of flat vector bundles and by shrinking $U$ we may assume that the corresponding map $G_U$ has the form
\[
G_U (y , \xi) = (y , \psi (\xi)) \quad \text{for} \; y \in U \; , \; \xi \in K^n \; , 
\]
where $\psi : K^n \to K^m$ is a linear map.

Set 
\[
Z_U = \{ \xi \in K^n \mid \varphi (\xi) = \psi (\xi) \} \; .
\]
They are the $K$-points of an algebraic variety in $\A^n_K$ and by the equalizer description of $P$ we have a homeomorphism
\[
P \, |_U \cong U \times Z_U \quad \text{over} \; U \; .
\]
It follows in particular that the function
\[
X \longrightarrow \{ 0 , 1 , \ldots , \infty \} \; , \; z \longmapsto |P_z|
\]
is locally constant and hence constant on $X$. Since $P_x = G_K (K)$ by \eqref{eq:21n} we conclude that $|P_z| \ge 1$ for all $z \in X$ i.e. that the projection $\pi : P \to X$ is surjective and that $Z_U \neq \emptyset$. Hence $P \, |_U$ has a section. Thus $P$ is a surjective $G_K (K)$-pseudo-torsor with local sections i.e. a $G_K (K)$-torsor.

\begin{theorem}
\label{t25n}
Let $(X,x)$ be a pointed connected topological space and $K$ a field.\\
1) For every algebraic quotient $\pi_K (X,x) \to G_K$ the corresponding $G_K (K)$-pseudo-torsor $P = P (G_K)$ over $X$ defined in \eqref{eq:20n} is a torsor. The algebraic group $G_K$ is reduced and we have $G_K = \overline{G_K (K)}$ in the Zariski topology.\\
2) For every flat vector bundle $E$ in $\uFl_K (X)$ the monodromy group $G_E$ is the Zariski closure with the reduced scheme structure in $\GL_{E_x}$ of the discrete group $G_E (K)$. The structure group of $E$ can be reduced to $G_E (K)$.\\
3) If the monodromy group scheme $G_E$ is finite, then $G_E$ is the constant group scheme over $K$ attached to the finite group $G_E (K)$.
\end{theorem}

\begin{rem}
The algebraic quotients $G_K$ of $\pi_K (X,x)$ are the quotients $\pi_K (X,x) \to G_E$ coming from the representations $\pi_K (X,x) \to \GL_{E_x}$ attached to flat vector bundles $E$ in $\uFl_K (X)$. This holds because every algebraic group scheme over $K$ admits a closed embedding into $\GL_H$ for a finite dimensional $K$-vector space $H$.
\end{rem}

\begin{proof}
1), 2) By the remark we may assume that $G_K = G_E$ for some $E$ in $\uFl_K (X)$. We have seen above that $P_E = P (G_E)$ is a $G_E (K)$-torsor. Using the $G_E (K)$-equivariant isomorphism \eqref{eq:24n} for $P = P (G_E)$ and the considerations after \eqref{eq:4} we obtain an isomorphism in $\uFl_K (X)$
\[
E = \pi^{G_E (K)}_* \pi^* E = P_E \times^{G_E (K)} E_x \; .
\]
Thus $E$ has a reduction of structure group to $G_E (K)$. Proposition \ref{t22} a) now implies that $G_E \subset \overline{G_E (K)}$ in $\GL_{E_x}$. Since $G_E$ is a closed subgroup scheme of $\GL_{E_x}$, we also have $\overline{G_E (K)} \subset G_E$ and therefore $G_E = \overline{G_E (K)}$. Such a group scheme is reduced. \\
Assertion 3) follows from 2).
\end{proof}

\begin{rem}
It follows from 2) that $\pi_K (X,x)$ is a projective limit of Zariski closures of discrete subgroups of $\GL_r (K)$ for varying $r$. In particular $\pi_K (X,x)$ is reduced. For locally connected spaces $X$ this was previously shown in Proposition \ref{t22} c) with a simpler proof.
\end{rem}

In the next section we will study the finite quotients of $\pi_K (X,x)$. For them the torsor $P$ is connected.

\begin{theorem}
\label{t26n}
Let $(X,x)$ be a pointed connected topological space and $K$ a field. Let $G_K$ be a finite group scheme quotient of $\pi_K (X,x)$. Then $P = P (G_K)$ is a connected torsor for the finite group $G_K (K)$ and $G_K = G_K (K) /_K$. 
\end{theorem}

\begin{proof}
We know from Theorem \ref{t25n}, that $P$ is a torsor. The following argument also gives another proof for this fact. Note that $A = A (G_K)$ defined after \eqref{eq:15n} is an algebra object in $\uFl_K (X)$ since $\Gamma (G_K)$ is a finite-dimensional $K$-vector space because $G_K$ is finite over $K$. Hence every point of $X$ has an open neighborhood such that in local coordinates for the flat bundle $A$ the multiplication map and the unit section are constant. Note here that for $r \ge 0$ the topology on $K^r$ is discrete. Consider the $G_K (K)$-pseudo-torsor
\[
\pi : P = P (G_K) = \uHom_{\alg} (A , \uK) \longrightarrow X \; .
\]
It follows that the function $z \mapsto |P_z|$ is locally constant and hence constant since $X$ is connected. We have
\[
P_x = \Hom_{\alg} (\Gamma (G_K ) , K) = G_K (K)
\]
and therefore $|P_z| = |G_K (K)| \ge 1$. Thus the map $\pi : P \to X$ is surjective. The above argument about local constancy also shows that $P$ is locally trivial. Thus $P$ is a $G_K (K)$-torsor. As in the proof of Theorem \ref{t25n} it follows that $G_K = \overline{G_K (K)}$. Since $G_K (K)$ is finite this implies that $G_K$ is the constant group scheme attached to the finite group $G_K (K)$. Writing $G = G_K (K)$ we therefore have $\Gamma (G_K) = K^G$. We claim that $A \cong \pi_* \uK$ for the projection $\pi : P \to X$. By \eqref{eq:4} ff it suffices to show that $\pi^* A$ and $\pi^* \pi_* \uK$ are isomorphic as $G$-bundles on the $G$-space $P$. The representation of $\pi_K (X,x)$ on $\Gamma (G_K)$ defining the bundle $A$ factors over the quotient $G_K$. Hence we may take $E = A$ in \eqref{eq:24n} and obtain a right $G$-equivariant isomorphism
\[
\varphi : \pi^* A \silo P \times A_x = P \times \Gamma (G_K) = P \times K^G \;.
\]
Here $\tau \in G$ acts on $P \times K^G$ via $(p,f)^{\tau} = (p^{\tau} , \tau^{-1} f)$ where $(\tau f) (\sigma) = f (\tau^{-1} \sigma)$. On the other hand we have a $G$-equivariant isomorphism 
\[
P \times K^G \silo \coprod_{p \in P} K^{p^G} = \coprod_{p \in P} K^{\pi^{-1} \pi (p)} = \pi^* \pi_* \uK \; .
\]
Here the element $(p,f)$ is sent to $(p^{\sigma} \mapsto f (\sigma))$. Thus $\pi^* A$ is isomorphic to $\pi^* \pi_* \uK$ as a $G$-bundle and hence we have $A = \pi_* \uK$ in $\uFl_K (X)$. We can now show that $P$ is connected. Namely
\begin{align*}
H^0 (P , \uK) & = H^0 (X , \pi_* \uK) = H^0 (X,A) = \Hom_{\Fl_K (X)} (\uK , A) \\
& \cong \Hom_{\uRep_K (\pi_K (X,x))} (K , K^G) \\
& = \Hom_{\uRep_K (G)} (K , K^G) = K \; .
\end{align*}
Here we have used that $G_K = G /_K$ is a quotient of $\pi_K (X,x)$ so that $\uRep_K (G)$ is a full subcategory of $\uRep_K (\pi_K (X,x))$. 
\end{proof}

For any two points $x_0 , x_1 \in X$ we define
\[ 
\pi_K (X , x_0 , x_1) = \uIso^{\otimes} (\omega_{x_0} , \omega_{x_1}) \; .
\]
By a general result on Tannakian categories \cite{DM}, Theorem 3.2, this is both a left $\pi_K (X , x_0)$- and a right $\pi_K (X, x_1)$-torsor for the $f\!pqc$-topology. 

We define the fundamental pro-algebraic groupoid $\Pi_K (X)$ of a topological space $X$ to be the following category enriched over the category of $K$-schemes. The objects of $\Pi_K (X)$ are the points of $X$. The morphism schemes are $\Mor (x_1 , x_2) = \emptyset$ if $x_1$ and $x_2$ lie in different connected components and 
\[
\Mor (x_1 , x_2) = \pi_K (Y , x_1 , x_2)
\]
if $x_1$ and $x_2$ lie in the same connected component $Y$ of $X$.

For a continuous map $f : X \to Y$ of topological spaces the pullback functor $f^* : \uFl_K (Y) \to \uFl_K (X)$ is a tensor functor and for any point $x \in X$, the diagram
\[
\xymatrix{
\uFl_K (Y) \ar[rr]^{f^*} \ar[dr]_{\omega_{f (x)}} & & \uFl_K (X) \ar[dl]^{\omega_x} \\
 & \uVec_K
}
\]
commutes. Hence we get an induced homomorphism of affine group schemes over $K$,
\[
f_* : \pi_K (X,x) \longrightarrow \pi_K (Y , f(y)) \; .
\]
More generally, for any two points $x_1 , x_2$ in $X$ we obtain a morphism (of bi-torsors)
\[
f_* : \pi_K (X , x_0 , x_1) \longrightarrow \pi_K (Y , f (x_0) , f (x_1)) \; .
\]
Let $\pi_1 (X , x)$ resp. $\pi_1 (X , x_1 , x_2)$ be the usual topological fundamental group resp. the (bi-torsor) of homotopy classes of continuous paths from $x_1$ to $x_2$. There is a natural homomorphism of groups
\begin{equation}
\label{eq:6}
\pi_1 (X,x) \longrightarrow \pi_K (X,x) (K)
\end{equation}
and more generally a morphism compatible with the bi-torsor structures
\begin{equation}
\label{eq:7}
\pi_1 (X , x_0 , x_1) \longrightarrow \pi_K (X , x_0 , x_1) (K)  \; .
\end{equation}
Namely, for a continuous path $\alpha : [0,1] \to X$ with $\alpha (0) = x_0$ and $\alpha (1) = x_1$ we obtain a $\otimes$-isomorphism $\omega_{\alpha} : \omega_{x_1} \silo \omega_{x_2}$ as follows: For $E$ in $\uFl_K (X)$ the locally constant sheaf $\alpha^{-1} \Eh$ is constant on $[0,1]$. Hence the evaluation maps $\ev_0$ and $\ev_1$ in the points $x_0$ and $x_1$ provide isomorphisms
\[
\omega_{x_0} (E) = E_{x_0} \xleftarrow{\overset{\ev_0}{\sim}} \Gamma ([0,1] , \alpha^* E) \xrightarrow{\overset{\ev_1}{\sim}} E_{x_1} = \omega_{x_1} (E) \; .
\]
The natural transformation $\omega_{\alpha}$ is defined by the family of isomorphisms
\[
\omega_{\alpha} (E) = \ev_1 \verk \ev^{-1}_0 : \omega_{x_0} (E) \longrightarrow \omega_{x_1} (E) \; .
\]
All local systems are constant on $[0,1] \times [0,1]$. Using this one sees that the isomorphisms $\omega_{\alpha} (E)$ and hence $\omega_{\alpha}$ depend only on the homotopy class of $\alpha$. 

It follows that any continuous path $\alpha$ from $x_0$ to $x_1$ defines an isomorphism of group schemes over $K$
\begin{equation}
\label{eq:8}
\alpha_K = R^{-1}_{\omega_{\alpha}} \verk L_{\omega_{\alpha}} : \pi_K (X , x_0) \silo \pi_K (X , x_1) \; .
\end{equation}
Here
\[
L_{\omega_{\alpha}} : \pi_K (X , x_0) \silo \pi_K (X , x_0 , x_1)
\]
and
\[
R_{\omega_{\alpha}} : \pi_K (X , x_1) \silo \pi_K (X , x_0 , x_1)
\]
are left- and right translation of $\omega_{\alpha} \in \pi_K (X , x_0 , x_1) (K)$.

If $X$ is not path-connected, I do not know if the group schemes $\pi_K (X , x_0)$ and $\pi_K (X ,x_1)$ are isomorphic over $K$ in general. For algebraically closed fields $K$ all fibre functors over $K$ of a Tannakian category over $K$ are isomorphic by a result of Deligne \cite{D}, see also \cite{Cou}, \cite{Wib}. Hence we have $\pi_K (X , x_0 , x_1) (K) \neq \emptyset$ if $K = \oK$ and for any element $\xi \in \pi_K (X , x_0 , x_1) (K)$ we get an isomorphism \eqref{eq:8} as above with $\omega_{\alpha}$ replaced by $\xi$.

We end this section with the following remark on flat vector bundles over a compact (= quasicompact + Hausdorff) space.

\begin{prop}
\label{t24}
Let $X$ be a compact connected topological space and let $E$ be in $\uFl_K (X)$. Then the following assertions hold:\\
a) There is a finite atlas of local trivializations of $E$ such that the finitely many transition functions
\[
g_{\nu \mu} : U_{\nu} \cap U_{\mu} \longrightarrow \GL_r (K) \; , \; r = \rank E
\]
take only finitely many different values.\\
b) There is a finitely generated field $K_0 \subset K$ and a flat vector bundle $E_0$ in $\uFl_{K_0} (X)$ such that $E \cong E \otimes_K K_0$ in $\uFl_K (X)$. 
\end{prop}

\begin{proof}
a) Choose a trivializing cover of $E$ with open sets $\{ V_i \}$ and let $\tg_{ij} : V_i \cap V_j \to \GL_r (K)$ be the corresponding locally constant transition functions. Choose a refining cover $\eU = \{ U_{\alpha} \}$ such that under the refining map $\iota$ we have $\oU_{\alpha} \subset V_{\iota (\alpha)}$. Then the induced transition maps
\[
g_{\alpha \beta} = \tg_{\iota (\alpha) \iota (\beta)} \; |_{U_{\alpha} \cap U_{\beta}}
\]
each take only finitely many values since $\oU_{\alpha} \cap \oU_{\beta} \subset V_{\iota (\alpha)} \cap V_{\iota (\beta)}$ is compact and $g_{\iota (\alpha) \iota (\beta)}$ is locally constant. By compactness of $X$ we may pass to a finite subcover $\{ U_{\nu} \}$ of $\{ U_{\alpha} \}$ and a) follows.\\
b) The entries of the finitely many matrices in $\GL_r (K)$ that occur as values of the $g_{\nu\mu}$ in a) generate a finitely generated subfield $K_0$ of $K$. Viewing $(g_{\nu\mu})$ as a cocycle with values in $\GL_r (K_0)$ we get a flat bundle $E_0$ such that $E_0 \otimes_{K_0} K$ is isomorphic to $E$. 
\end{proof}

\begin{remark}
\label{t25}
The proof of a) shows that the finitely many transition functions $g_{\nu\mu}$ take values in $\GL_r (A_0)$ for a finitely generated $\Z$-algebra. Reducing modulo a maximal ideal $\emm_0$ of $A_0$ one obtains a cocycle with values in $\GL_r (k)$ where $k = A_0 / \emm_0$ is a finite field. The associated principal bundle is a finite covering of $X$. Its connected components are finite Galois coverings, c.f. \cite{KS}, Proposition 2.7. 
\end{remark}

\section{Behaviour of $\pi_K (X,x)$ under finite coverings and relation with $\pi^{\et}_1 (X,x)$} \label{sec:3}
Let $\pi : Y \to X$ be a finite covering of the topological space $X$. The map $X \to \Z , x \mapsto |\pi^{-1} (x)|$ is locally constant and hence constant on each connected component of $X$. If $X$ is connected its value is called the degree $\deg (\pi)$ of the covering. In \cite{KS} Proposition 2.7 it is shown that for a connected topological space $X$ the total space of every finite covering $\pi : Y \to X$ has only finitely many connected components $Y = Y_1  \amalg \ldots \amalg Y_r$. Moreover, the restrictions $\pi_i = \pi \, |_{Y_i}$ are finite coverings and we have $\deg \pi = \deg \pi_1 + \ldots + \deg \pi_r$. A finite Galois covering with group $G$ is a finite connected covering $\pi : Y \to X$ together with a (right-) $G$-action on $Y$ over $X$ such that $G$ permutes the fibres of $\pi$ simply transitively. Equivalently it is a connected $G$-torsor over $X$. Every bundle $F$ in $\uFl_K (Y)$ is a subbundle of $\pi^* E$ for some $E$ in $\uFl_K (X)$. Namely, for the flat bundle (!) $E = \pi_* F$ we have 
\[
\pi^* E = \pi^* \pi_* F = \bigoplus_{\sigma \in G} \sigma^* F \; .
\]
Thus $F$ is even a direct summand of $\pi^* E$. By \cite{DM}, Proposition 2.21 (b), for any chosen $y \in Y$ with $\pi (y) = x$ the morphism 
\[
i = \pi_* : \pi_K (Y, y) \longrightarrow \pi_K (X,x)
\]
is therefore a closed immersion.

Assume that $F$ in $\uFl^G_K (Y)$ is a trivial bundle in $\uFl_K (Y)$. Then there is an isomorphism of flat bundles, where $V$ is a finite dimensional $K$-vector space with the discrete topology
\[
\varphi : F \silo Y \times V \; .
\]
The right $G$-action on $F$ over the $G$-space $Y$ gives a right $G$-action on $Y \times V$ over the $G$-space $Y$. Since $Y$ is connected being a Galois covering, there is a representation $\rho : G \to \GL (V)$ such that we have:
\begin{equation}
\label{eq:34n}
(y , v)^{\sigma} = (y^{\sigma} , \rho (\sigma^{-1}) v) \quad \text{for} \; \sigma \in G \; , \; y \in Y \; , \; v \in V \; .
\end{equation}
Let $\uFl_K (X) (\pi)$ be the full subcategory of bundles $E$ in $\uFl_K (X)$ for which $\pi^* E$ is a trivial bundle in $\uFl_K (Y)$. Using \eqref{eq:4}, \eqref{eq:34n} and the discussion of morphisms between trivial bundles on a connected covering space in the proof of Proposition \ref{t22}, we obtain the following fact.

\begin{prop} \label{t31}
1) The functor
\[
\uRep_K (G) \longrightarrow \uFl_K (X) (\pi) \; , \; V \longmapsto \pi^G_* (Y \times V) = Y \times^G V
\]
is an equivalence of categories.\\
2) We have
\[
\pi^G_* (Y \times V)_x = \pi^{-1} (x) \times^G V \; .
\]
Fixing a point $y \in Y$ over $x$ this is canonically isomorphic to $V$. Using 1) we get an isomorphism (depending on $y$) between $G /_K$ and the Tannakian dual $\pi_K (X,x) (\pi)$ of $(\uFl_K (X) (\pi), \omega_x)$. 
\end{prop}

Since $\pi^*$ is exact and because of Corollary \ref{t23} any subobject in $\uFl_K (X)$ of an object in $\uFl_K (X) (\pi)$ lies in $\uFl_K (X) (\pi)$. It follows from \cite{DM} Proposition 2.21 that the induced morphism
\[
p  : \pi_K (X, x) \longrightarrow \pi_K (X, x) (\pi)
\]
is faithfully flat.

\begin{prop} \label{t32}
Let $\pi : Y \to X$ be a finite covering with group $G$. Choose a point $y \in Y$ and set $x = \pi (y)$. The following sequence of group schemes over $K$ is exact:
\[
1 \longrightarrow \pi_K (Y, y) \xrightarrow{i} \pi_K (X,x) \xrightarrow{p} \pi_K (X,x) (\pi) \longrightarrow 1 \, .
\]
The choice of $y$ over $x$ determines a canonical isomorphism $\pi_K (X,x) (\pi) = G /_K$ by Proposition \ref{t31}. 
\end{prop}

\begin{proof}
It remains to show that $i$ is an isomorphism of $\pi_K (Y,y)$ onto $\ker i$. In \cite{EHS} Theorem A1 (iii), a Tannakian criterion is given for this. Translated to our context the following three conditions need to be verified:\\
a) For a bundle $E$ in $\uFl_K (X)$ the bundle $\pi^* E$ is trivial if and only if $E$ is in $\uFl_K (X) (\pi)$. \\
b) Let $F_0$ be the maximal trivial subbundle of $\pi^* E$ in $\uFl_K (Y)$. Then there is a subbundle $E_0 \subset E$ in $\uFl_K (X)$ such that $F_0 \cong \pi^* E_0$.\\
c) Every bundle $F$ in $\uFl_K (Y)$ is a subbundle of $\pi^* E$ for some $E$ in $\uFl_K (X)$.\\
Condition a) is true by the definition of $\uFl_K (X) (\pi)$. Condition c) has already been verified. The maximal trivial subbundle $F_0$ of $\pi^* E$ exists. It is the subbundle generated by the global sections $\Gamma (Y , \pi^* E)$. This description also shows that it is a $G$-subbundle of the $G$-bundle $\pi^* E$ on the $G$-space $Y$. Using the equivalence of categories \eqref{eq:4} and setting
\[
E_0 := \pi^G_* F_0 \subset \pi^G_* \pi^* E \cong E \; ,
\]
we find
\[
F_0 \cong \pi^* \pi^G_* F_0 = \pi^* E_0
\]
as desired. 
\end{proof}

The finite coverings with the obvious morphisms form a category $\uFCov (X)$. According to \cite{KS} Proposition 2.9, $\uFCov (X)$ is a Galois category in the sense of \cite{SGA1} Expos\'e V.4 if $X$ is connected. For any point $x \in X$ there is a natural fibre functor
\[
\Phi_x : \uFCov (X) \longrightarrow \uFSet
\]
into the category of finite sets which on objects is given by
\[
\Phi_x (\pi : Y \to X) = \pi^{-1} (x) \; .
\]
For a connected topological space $X$ and a point $x \in X$, Kucharczyk and Scholze define the \'etale fundamental group $\pi^{\et}_1 (X,x)$ to be the automorphism group of $\Phi_x$. More generally, for points $x_1 , x_2 \in X$ we set
\[
\pi^{\et}_1 (X , x_1 , x_2) = \Iso (\Phi_{x_1} , \Phi_{x_2}) \; .
\]
It is a non-empty profinite set by \cite{SGA1} Expos\'e V, Corollaire 5.7. Moroever, $\pi^{\et}_1 (X , x_1 , x_2)$  is a left- resp. right torsor under the profinite groups $\pi^{\et}_1 (X , x_1)$ resp. $\pi^{\et}_1 (X , x_2)$. The \'etale fundamental groupoid $\Pi^{\et}_1 (X)$ of a topological space $X$ is defined as the small topological category whose objects are the points of $X$ and for which $\Mor (x_1 , x_2) = \emptyset$ if $x_1$ and $x_2$ lie in different connected components of $X$ and
\[
\Mor (x_1 , x_2) = \pi^{\et}_1 (Z , x_1 , x_2)
\]
if $x_1$ and $x_2$ lie in the same connected component $Z$ of $X$. 

We will now relate $\pi_K (X , x_1 , x_2)$ to $\pi^{\et}_1 (X , x_1 , x_2)$. Namely let $\uFFl_K (X)$ be the full subcategory of $\uFl_K (X)$ whose objects are the bundles $E$ for which there exists a finite covering $\pi : Y \to X$ such that $\pi^* E$ is a trivial flat bundle i.e. isomorphic in $\uFl_K (Y)$ to a bundle of the form $Y \times K^r , r \ge 0$. Given two finite coverings $\pi_1 : Y_1 \to X$ and $\pi_2 : Y_2 \to X$ the fibre product $\pi : Y = Y_1 \times_X Y_2 \to X$ is a finite covering which factors over $\pi_1$ and $\pi_2$. It follows that if $E_1 , E_2$ are in $\uFFl_K (X)$ then $E_1 \oplus E_2 , E_1 \otimes E_2$ and the $\Hom$ bundle $\uHom (E_1 , E_2)$ are in $\uFFl_K (X)$ as well. Hence for connected $X$ the category $\uFFl_K (X)$ is neutral Tannakian with fibre functor $\omega^F_x$ defined as the restriction of $\omega_x$ to $\uFFl_K (X)$. Let $\pi^{\et}_K (X,x)$ be the Tannakian dual group of $(\uFFl_K (X) , \omega^F_x)$
\[
\pi^{\et}_K (X , x) = \uAut^{\otimes} (\omega^F_x) \; .
\]
More generally, for points $x_1 , x_2 \in X$ we set
\[
\pi^{\et}_K (X , x_1 , x_2) = \uIso (\omega^F_{x_1} , \omega^F_{x_2}) \; .
\]
We define $\Pi^{\et}_K (X)$ similarly as $\Pi_K (X)$. Passing to the $K$-valued points of the morphism schemes we obtain a topological category $\Pi^{\et}_K (X) (K)$. 

We will now define a parallel transport along (``homotopy classes'' of) \'etale paths $\gamma$ i.e. elements of $\pi^{\et}_1 (X , x_1, x_2)$ for bundles $E$ in $\uFFl_K (X)$. Choose a finite connected covering $\pi : Y \to X$ such that $\pi^* E$ is trivial. Let $y_1 \in Y$ be a point with $\pi (y_1) = x_1$. Then $\gamma (y_1) \in Y$ is a point with $\pi (\gamma (y_1)) = x_2$. Since $\pi^* E$ is trivial and $Y$ is connected the evaluation maps $\ev_{y_1}$ and $\ev_{\gamma (y_1)}$ are isomorphisms:
\begin{equation}
\label{eq:10}
E_{x_1} = (\pi^* E)_{y_1} \xleftarrow{\overset{\ev_{y_1}}{\sim}} \Gamma (Y , \pi^* E) \xrightarrow{\overset{\ev_{\gamma (y_1)}}{\sim}} (\pi^* E)_{\gamma (y_1)} = E_{x_2} \; .
\end{equation}
Let 
\[ \rho_{\gamma} (E) : E_{x_1} \silo E_{x_2}
\]
be the resulting isomorphism.

\begin{theorem}
\label{t33}
1) For a connected topological space $X$ and points $x_1 , x_2 \in X$ the isomorphisms $\rho_{\gamma} (E)$ for $E$ in $\uFFl_K (X)$ are well defined and give rise to an isomorphism of the fibre functors $\omega^F_{x_1}$ and $\omega^F_{x_2}$ over $K$, i.e. to an element $\rho_{\gamma}$ of $\pi^{\et}_K (X , x_1 , x_2) (K)$.\\
2) The resulting functor
\[
\rho : \Pi^{\et}_1 (X) \longrightarrow \Pi^{\et}_K (X) (K)
\]
is an isomorphism of topological categories. In particular the maps
\[
\rho : \pi^{\et}_1 (X , x_1 , x_2) \silo \pi^{\et}_K (X , x_1 , x_2) (K) \; , \; \gamma \mapsto \rho_{\gamma}
\]
are homeomorphisms of pro-finite spaces, and for $x \in X$ the maps
\[
\rho : \pi^{\et}_1 (X, x) \silo \pi^{\et}_K (X,x) (K) \; , \; \gamma \mapsto \rho_{\gamma}
\]
are topological group isomorphisms.
\end{theorem}

\begin{proof}
1) We first prove that $\rho_{\gamma} (E)$ is independent of the choice of the point $y_1$ over $x_1$. There is a finite Galois covering $\tY \xrightarrow{\tpi} X$ which factors over $\pi$. Pulling back to $\tY$ and writing down obvious commutative diagrams one sees that to prove independence of $y_1$ over $x_1$ we may pass to $\tY$ and thus assume that $Y$ is Galois with group $G$. Let $y'_1 \in Y$ be another point over $x$. Choose an element $\sigma \in G$ with $y^{\sigma}_1 = y'_1$. Noting that $(\gamma y_1)^{\sigma} = \gamma (y^{\sigma}_1) = \gamma y'_1$, we obtain the commutative diagram
\[
\xymatrix{
E_{x_1} \ar@{=}[r] & (\pi^* E)_{y'_1} \ar[d]^{\wr}_{\sigma^*} & \Gamma (Y , \pi^* E) \ar[l]_{\overset{\ev_{y'_1}}{\sim}} \ar[d]^{\wr}_{\sigma^*} \ar[r]^{\overset{\ev_{\gamma y'_1}}{\sim}} & (\pi^* E)_{\gamma  y'_1} \ar[d]^{\wr}_{\sigma^*} \ar@{=}[r] & E_{x_2} \\
& (\sigma^* \pi^* E)_{y_1} \ar@{=}[d] & \Gamma (Y , \sigma^* \pi^* E) \ar[l]_{\overset{\ev_{y_1}}{\sim}} \ar[r]^{\overset{\ev_{\gamma_{y_1}}}{\sim}} \ar@{=}[d] & (\sigma^* \pi^* E)_{\gamma y_1} \ar@{=}[d] & \\
E_{x_1} \ar@{=}[r] & (\pi^* E)_{y_1} & \Gamma (Y , \pi^* E) \ar[l]_{\overset{\ev_{y_1}}{\sim}} \ar[r]^{\overset{\ev_{\gamma y_1}}{\sim}} & (\pi^* E)_{\gamma  y_1} \ar@{=}[r] & E_{x_2}
}
\]
It follows that
\[ 
\rho_{\gamma} (E) = \ev_{\gamma y_1} \verk \ev^{-1}_{y_1} = \ev_{\gamma y'_1} \verk \ev^{-1}_{y'_1}
\]
is independent of the choice of $y_1$ over $x_1$. Independence of the connected finite covering trivializing $E$ follows by dominating two such coverings $Y_1 \to X$ and $Y_2 \to X$ by a third one e.g. by a connected component of $Y_1 \times_X Y_2$. Thus the isomorphism $\rho_{\gamma} (E)$ is well defined. Elementary arguments show that the family $(\rho_{\gamma} (E))$ for $E$ in $\uFFl_K (X)$ defines an isomorphism from the $\otimes$-functor $\omega^F_{x_1}$ to the $\otimes$-functor $\omega^F_{x_2}$. Thus one obtains an element 
\[
\rho_{\gamma} \in \uIso^{\otimes} (\omega^F_{x_1} , \omega^F_{x_2}) (K) \; .
\]
By the construction of the parallel transport $\rho_{\gamma} (E)$ it is clear that for $\gamma \in \pi^{\et}_1 (X , x_1, x_2)$ and $\gamma' \in \pi^{\et}_1 (X , x_2 , x_3)$ we have
\[
\rho_{\gamma' \verk \gamma} (E) = \rho_{\gamma'} (E) \verk \rho_{\gamma} (E)
\]
and hence
\[
\rho_{\gamma' \verk \gamma} = \rho_{\gamma'} \verk \rho_{\gamma} \; .
\]
It follows that
\[
\rho : \Pi^{\et}_1 (X) \to \Pi^{\et}_K (X) (K)
\]
is a functor. In particular the map
\[
\rho : \pi^{\et}_1 (X, x) \longrightarrow \pi^{\et}_K (X , x) (K)
\]
is a homomorphism of groups for all $x \in X$. \\
2) Given a finite Galois covering $\pi : Y \to X$ with group $G$, by construction the parallel transport on bundles in $\uFl_K (X) (\pi)$ along $\gamma \in \pi^{\et}_1 (X , x_1 , x_2)$ depends only on the bijection
\[
\gamma (Y) : F_{x_1} (Y) = \pi^{-1} (x_1) \longrightarrow \pi^{-1} (x_2) = F_{x_2} (Y) \; .
\]
Hence we get a map
\begin{equation}
\label{eq:11}
\Imm (\pi^{\et}_1 (X , x_1 , x_2) \longrightarrow \Bij (\pi^{-1} (x_1) , \pi^{-1} (x_2)) \longrightarrow (\pi_K (X , x_1 , x_2) (\pi)) (K) \; .
\end{equation}
Here
\[
\pi_K (X , x_1 , x_2) (\pi) = \uIso^{\otimes} (\omega^{\pi}_{x_1} , \omega^{\pi}_{x_2})
\]
where $\omega^{\pi}_x$ is the restriction of $\omega_x$ to $\uFl_K (X) (\pi)$. 

We claim that \eqref{eq:11} is a bijection. In order to show that \eqref{eq:11} is injective we need to show that for each $\gamma$ the bijection $\gamma (Y)$ is uniquely determined by the parallel transport $\rho_{\gamma}$ on bundles $E$ in $\uFl_K (X) (\pi)$. Consider $A = \pi_* \uK$ in $\uFl_K (X) (\pi)$ and choose a point $y_1 \in Y$ over $x_1$. By construction
\[
\rho_{\gamma} (A) : A_{x_1} = K^{\pi^{-1} (x_1)} \silo K^{\pi^{-1} (x_2)} = A_{x_2}
\]
sends $\delta^1_y$ to $\delta^2_{\gamma_y}$. Here $\delta^1_y : \pi^{-1} (x_1) \to K$ is $= 1$ on $y$ and $ = 0$  on all other points, and $\delta^2_{\gamma y} : \pi^{-1} (x_2) \to K$ is $= 1$ on $\gamma y$ and $= 0$ on the other points. Hence we recover $\gamma (Y)$, the image of $\gamma$ in $\Bij (\pi^{-1} (x_1) , \pi^{-1} (x_2))$ uniquely from $\rho_{\gamma} (A)$ and hence from the image of $\rho_{\gamma}$ in $\pi_K (X , x_1 , x_2) (\pi) (K)$. Since $X$ is connected, $\pi^{\et}_1 (X , x_1 , x_2)$ is not empty, \cite{SGA1} V, Corollaire 5.7. Hence both the source and the target of \eqref{eq:11} are non-empty sets. Since they are principal homogenous spaces, for bijectivity of \eqref{eq:11} it suffices to show that the group homomorphism
\begin{equation}
\label{eq:12}
\Imm (\pi^{\et}_1 (X , x) \longrightarrow \Bij (\pi^{-1} (x))) \longrightarrow (\pi_K (X,x) (\pi)) (K)
\end{equation}
is an isomorphism for any $x \in X$. We have seen that it is injective. Since $\pi : Y \to X$ is a Galois covering with group $G$ the source is isomorphic to $G$. By Proposition \ref{t31} the affine group $\pi_K (X,x) (\pi)$ is isomorphic to $G /_K$. Hence the target is isomorphic to $G$ as well. It follows that \eqref{eq:12} is an isomorphism. One can see surjectivity of \eqref{eq:12} also directly by studying $\alpha \in (\pi_K (X , x) (\pi)) (K) = \Aut^{\otimes} (\omega^{\pi}_x)$ on the $\otimes$-generator $A = \pi_* \uK$ of $\uFl_K (X) (\pi)$ and noting that $A$ is a bundle of $K$-algebras with a $G$-operation. Compatibility of $\alpha$ with the multiplication $A \otimes A \to A$ and the $G$-action show that $\alpha (A)$ and hence $\alpha$ come from the source of \eqref{eq:12}. 

Taking the projective limit over all finite Galois coverings $\pi : Y \to X$ of the bijections \eqref{eq:11} we obtain a homeomorphism of pro-finite spaces
\[
\rho : \pi^{\et}_1 (X , x_1 , x_2) \silo \pi^{\et}_K (X , x_1 , x_2) (K) \; .
\]
The remaining assertions follow immediately.
\end{proof}

\begin{prop}
\label{t34}
Let $X$ be a connected topological space, $x \in X$ and $G$ a finite group. A vector bundle $E$ in $\uFl_K (X)$ is trivialized by a finite Galois covering $\pi : Y \to X$ with group $G$ if and only if there is a faithfully flat morphism $\pi_K (X,x) \twoheadrightarrow G /_K$ such that the monodromy representation of $E$ factors
\begin{equation}\label{eq:13nn}
\pi_K (X,x) \twoheadrightarrow G/_K \twoheadrightarrow G_E \subset \GL_{E_x} \; .
\end{equation}
In this case we have $E \cong Y \times^G E_x$ in $\uFl_K (X)$ and the monodromy group $G_E$ of $E$ is a constant group scheme $G_E = G_E (K) /_K$ where $G_E (K)$ is a quotient of $G$.
\end{prop}

\begin{proof}
If $E$ is trivialized by $\pi : Y \to X$ all claims follow from Propositions \ref{t31} and \ref{t32}. Now assume that we have a factorization \eqref{eq:13nn}. Applying Theorem \ref{t26n} to the quotient $G_K = G /_K$ of $\pi_K (X,x)$ we get a connected torsor $\pi: P \to X$ for the group $G_K (K) = G$, in other words a finite Galois covering with group $G$. According to Theorem \ref{t24n}, 1) the bundle $\pi^* E$ is trivial in $\uFl_K (P)$. 
\end{proof}

\begin{cor}
\label{t35}
Let $X$ be a connected topological space, $x \in X$, $K$ a field and $E$ a vector bundle in $\uFl_K (X)$ of rank $r$. Then $E$ has a reduction of structure group to a finite subgroup $G$ of $\GL_r (K)$ if and only if the monodromy group $G_E$ is a constant finite group scheme. In this case the following is true:\\
a) $G_E$ is a subquotient of $G /_K$.\\
b) Choosing a basis of $E_x$ and viewing $G_E (K)$ as a subgroup of $\GL_r (K)$, the structure group of $E$ can be reduced to $G_E (K)$, and up to conjugacy in $\GL_r (K)$ this is the smallest subgroup of $\GL_r (K)$ for which this is possible.
\end{cor}

\begin{rem}
In Theorem \ref{t25n}, 3) or Theorem \ref{t26n} it was shown that if $G_E$ is a finite group scheme over $K$, then it is constant.
\end{rem}

\begin{proof}
If the structure group of $E$ can be reduced to a finite group $G$ in $\GL_r (K)$, there is a principal $G$-bundle $P \xrightarrow{\pi} X$ such that $E \cong P \times^G K^r$ and hence $\pi^* E$ is a trivial bundle in $\uFl_K (P)$. Let $P_0$ be a connected component of the finite covering $P$ and let $G_0 \subset G$ consist of $\sigma \in G$ with $\sigma (P_0) = P_0$. Then $\pi \, |_{P_0} : P_0 \to X$ is a Galois covering which trivializes $E$. Proposition \ref{t34} implies that $G_E$ is a constant finite group scheme over $K$ and that $G_E (K)$ is a quotient of $G_0$ and hence a subquotient of $G$. If on the other hand, $E$ has a constant monodromy group $G_E$, then by Proposition \ref{t34}, taking $G = G_E (K)$ there is a Galois covering $\pi : Y \to X$ with $E \cong Y \times^G E_x$. This description of $E$ shows that viewing $G_E (K)$ as a subgroup of $\GL_r (K)$ (unique up to conjugation) via an isomorphism $E_x \cong K^r$, the structure group of $E$ can be reduced to $G_E (K)$. The remaining assertions follow.
\end{proof}

\begin{exmp}
In Remark \ref{t25} we have seen that on a compact connected space $X$ any flat vector bundle $E$ in $\uFl_K (X)$ of rank $r$ has a reduction of structure group to $\GL_r (A_0)$ where $A_0$ is a finitely generated $\Z$-algebra in $K$. Reducing modulo a maximal ideal $\emm_0$ of $A_0$ we obtain a vector bundle $E_{\emm_0}$ in $\uFl_{k_0} (X)$ where $k_0 = A_0 / k_0$ is a finite field. The monodromy group of $E_{\emm_0}$ is contained in $\GL_r (k_0)$ and is therefore finite. The construction in the proof of Proposition \ref{t34} attaches a Galois covering $\pi : P \to X$ to $E_{\emm_0}$ with Galois group the monodromy group of $E_{\emm_0}$ in $\GL_r (k_0)$. 
\end{exmp}

Let $H = \spec B$ be an affine group scheme over a field $K$. The largest separable subalgebra $B^{\et}$ of $B$ is a Hopf algebra and $H^{\et} = \spec B^{\et}$ is a pro-\'etale group scheme over $K$. It is the maximal pro-\'etale quotient of $H$, any morphism of $H$ to an \'etale group factors uniquely over the faithfully flat projection $H \to H^{\et}$. There is a natural exact sequence
\[ 
1 \longrightarrow H^0 \longrightarrow H \longrightarrow H^{\et} \longrightarrow 1 \; ,
\]
where $H^0$ is the connected component of the identity. Thus $H^{\et}$ may also be viewed as the group scheme of connected components of $H$ and there is the alternative notation $H^{\et} = \pi_0 (H)$. A pointed finite (Galois) covering $(Y, y) \to (X,x)$ is a finite (Galois) covering $\pi_Y : Y \to X$ with $\pi_Y (y) = x$. Morphisms of such pointed coverings are defined in the evident way. They are unique if they exist, in case the coverings are connected. The set of isomorphism classes of finite pointed Galois coverings of $(X,x)$ is a directed set $\Gh = \Gh (X,x)$ if $(Y_1 , y_1) \ge (Y_2 , y_2)$ means that there is a morphism $(Y_1 , y_1) \to (Y_2 , y_2)$, c.f. \cite{KS} section 2.

We have
\[
\uFFl_K (X) = \varinjlim_{(Y,y) \in \Gh} \uFl_K (X) (\pi_Y)
\]
in $\uFl_K (X)$ and therefore
\begin{equation}
\label{eq:26}
\pi^{\et}_K (X , x) = \varprojlim_{(Y,y) \in \Gh} \pi_K (X,x) (\pi_Y) \; .
\end{equation}
This is a projective limit of finite constant group schemes over $K$, since $\pi_K (X , x) (\pi)$ is finite constant by Proposition \ref{t31}. In particular the natural isomorphism of ``parallel transport along closed loops'' of Theorem \ref{t33}
\[
\rho : \pi^{\et}_1 (X,x) \silo \pi^{\et}_K (X,x) (K)
\]
can be viewed as an isomorphism of pro-finite-constant group schemes over $K$
\begin{equation}
\label{eq:27}
\pi^{\et}_1 (X,x) /_K \silo \pi^{\et}_K (X,x) \; .
\end{equation}
Since $\uFFl_K (X)$ is a full subcategory of $\uFl_K (X)$ which is closed under taking subobjects in $\uFl_K (X)$ by Corollary \ref{t23}, the induced morphism 
\[
\pi_K (X,x) \longrightarrow \pi^{\et}_K (X,x)
\]
is faithfully flat. It factors over the maximal pro-\'etale quotient
\[
\pi_K (X,x)^{\et} = \pi_0 (\pi_K (X,x))
\]
of $\pi_K (X,x)$. Thus we obtain a faithfully flat morphism of group schemes over $K$
\begin{equation}
\label{eq:28}
\pi_K (X,x)^{\et} \longrightarrow \pi^{\et}_K (X,x) \; .
\end{equation}

\begin{theorem}
\label{t36}
Let $(X,x)$ be a pointed connected topological space and $K$ a field. Then \eqref{eq:28} is an isomorphism
\begin{equation}
\label{eq:29}
\pi_K (X,x)^{\et} \silo \pi^{\et}_K (X,x) \; .
\end{equation}
Using the isomorphisms \eqref{eq:27} and \eqref{eq:29} we have an exact sequence
\begin{equation}
\label{eq:30}
1 \longrightarrow \pi_K (X,x)^0 \longrightarrow \pi_K (X,x) \longrightarrow \pi^{\et}_1 (X,x) /_K \longrightarrow 1 \; .
\end{equation}
Moreover, there is a natural isomorphism
\begin{equation}
\label{eq:31}
\pi_K (X,x)^0 = \varprojlim_{(Y,y) \in \Gh (X,x)} \pi_K (Y,y) \; .
\end{equation}
\end{theorem}

\begin{proof}
Any finite \'etale quotient of $\pi_K (X,x)$ is a constant group scheme by Theorem \ref{t25n}, 3) or Theorem \ref{t26n}. Using \eqref{eq:26} and Proposition \ref{t34} it follows that \eqref{eq:28} is an isomorphism since both sides have the same algebraic representations. Hence we obtain \eqref{eq:29} and the exact sequence \eqref{eq:30}. Passing to the limit in the exact sequence of Proposition \ref{t32} and using \eqref{eq:26} and \eqref{eq:27} we obtain an exact sequence
\begin{equation}
\label{eq:34}
1 \longrightarrow \varprojlim_{(Y,y) \in \Gh} \pi_K (Y,y) \longrightarrow \pi_K (X,x) \longrightarrow \pi^{\et}_1 (X,x) /_K \; .
\end{equation}
Comparing \eqref{eq:34} and \eqref{eq:30} the isomorphism \eqref{eq:31} follows. 
\end{proof}

The proof of Theorem \ref{t36} only needed the special case of Theorem \ref{t25n}, 3) or Theorem \ref{t26n} that all finite {\it \'etale} quotients of $\pi_K (X,x)$ are constant. This can also be shown using the following proposition which is of indepenent interest. 

\begin{prop}
\label{t37}
Let $X$ be a connected topological space and $L / K$ a field extension. Let $E$ be a flat vector bundle in $\uFl_K (X)$ such that $E \otimes_K L$ is a trivial bundle in $\uFl_L (X)$. Then $E$ is a trivial bundle in $\uFl_K (X)$.
\end{prop}

\begin{proof}
By assumption there exists an open cover $\eU = (U_i)_{i \in I}$, a representing cocycle $g = (g_{ij})$ of locally constant maps $g_{ij} : U_i \cap U_j \to \GL_r (K)$ for the isomorphism class of $E$ and locally constant maps
\[
l_i : U_i \longrightarrow \GL_r (L) \quad \text{for} \; i \in I
\]
such that
\[
g_{ij} = l^{-1}_i l_j \quad \text{on} \; U_i \cap U_j \; \text{for} \; i , j \in I \; .
\]
Consider the composition
\[
\ol_i : U_i \xrightarrow{l_i} \GL_r (L) \xrightarrow{\text{---}} \GL_r (L) / \GL_r (K) \; .
\]
We have
\[
\ol_i = \overline{l_i g_{ij}} = \ol_j \quad \text{on} \; U_i \cap U_j \; .
\]
Thus the maps $\ol_i$ glue to a locally constant map
\[
\ol : X \longrightarrow \GL_r (L) / \GL_r (K) \; .
\]
Since $X$ is connected, $\ol$ is constant, $\ol = a \GL_r (K)$ for a matrix $a \in \GL_r (L)$. Thus we have $l_i = ag_i$ on $U_i$ for $i \in I$ with locally constant maps
\[
g_i : U_i \longrightarrow \GL_r (K) \; .
\]
This implies that
\[
g_{ij} = (ag_i)^{-1} (ag_j) = g^{-1}_i g_j \; .
\]
Hence $E$ is isomorphic to a trivial bundle in $\uFl_K (X)$.
\end{proof}

The required special case of Theorem \ref{t25n}, 3) or Theorem \ref{t26n} needed for the proof of Theorem \ref{t36} is the following assertion for which we give a direct proof. 

\begin{prop}
\label{t38}
Let $(X,x)$ be a pointed connected topological space and $K$ a field. Let $E$ in $\uFl_K (X)$ be a flat vector bundle whose monodromy group scheme $G_E = G_{E,x}$ is finite \'etale. Then $G_E$ is a constant group over $K$.
\end{prop}

\begin{proof}
For a field extension $L / K$, write $\omega^L_x$ for the fibre functor $F \mapsto F_x$ on $\uFl_L (X)$. Recall the functor $\phi_L : \uVec_K \to \uVec_L , V \mapsto V \otimes_K L$. The $\otimes$-functor 
\[
\uFl_K (X) \to \uFl_L (X) , E \mapsto E \otimes_K L
\]
is compatible with the fibre functors $\phi_L \verk \omega_x$ on the left and $\omega^L_x$ on the right. Hence we get a morphism of group schemes over $L$
\begin{equation}
\label{eq:32}
\pi_L (X, x) = \uAut^{\otimes} (\omega^L_x) \longrightarrow \uAut^{\otimes} (\phi_L \verk \omega_x) = \uAut^{\otimes} (\omega_x) \otimes_K L = \pi_K (X,x) \otimes_K L \; .
\end{equation}
Let $\rho : \pi_K (X,x) \twoheadrightarrow G_E \subset \GL_{E_x}$ be the representation corresponding to $E$. Then the composition
\[
\rho_L : \pi_L (X,x) \longrightarrow \pi_K (X,x) \otimes_K L \overset{\rho \otimes L}{\twoheadrightarrow} G_E \otimes_K L \subset \GL_{E_x \otimes_K L}
\]
is the representation corresponding to $E \otimes_K L$ in $\uFl_L (X)$. Hence we have a closed immersion
\begin{equation}
\label{eq:33}
G_{E \otimes_K L} \subset G_E \otimes_K L
\end{equation}
of closed subgroup schemes of $\GL_{E_x \otimes_K L}$. By assumption there is a field extension $L / K$ e.g. $L = K^{\sep}$, such that $G_E \otimes_K L$ is constant. Because of \eqref{eq:33} the affine group $G_{E \otimes_K L}$ is therefore constant as well. Hence, by Proposition \ref{t34} there is a finite Galois covering $\pi : Y \to X$ such that $\pi^* (E \otimes_K L) = (\pi^* E) \otimes_K L$ is a trivial bundle in $\uFl_L (Y)$. Proposition \ref{t37} now implies that $\pi^* E$ is trivial in $\uFl_K (Y)$. Invoking Proposition \ref{t34} again, it follows that $G_E$ is constant over $K$.
\end{proof}
\section{Calculations of $\pi_K (X,x)$} \label{sec:4}
Given an abstract group $\Gamma$ the proalgebraic completion of $\Gamma$ over the field $K$ is a pair consisting of an affine group scheme $\Gamma^{\alg} = \Gamma^{\alg}_K$ over $K$ and a homomorphism of groups $i : \Gamma \to \Gamma^{\alg} (K)$ with Zariski dense image. It is defined up to unique automorphism by the following universal property: For any representation $\rho : \Gamma \to \GL(V)$ on 
a finite dimensional $K$-vector space $V$ there is a unique algebraic representation $\rho^{\alg} : \Gamma^{\alg} \to \GL_V$ with $\rho = \rho^{\alg} (K) \verk i$. One can obtain $\Gamma^{\alg}$ as the Tannakian dual of the neutral Tannakian category $\Rep_K (\Gamma)$ of finite dimensional $K$-representations of $\Gamma$ with respect to the fibre functor of forgetting the $\Gamma$-action. A concrete description of the Hopf-algebra $A$ over $K$ with $\Gamma^{\alg} = \spec A$ is the following. The group $\Gamma$ acts on the $K$-algebra of function $f : \Gamma \to K$ by right and left translation. The algebra $A$ consists of all functions whose left (equiv. right) $\Gamma$-orbits generate finite-dimensional $K$-vector spaces. The comultiplication, co-inverse and co-unit are obtained by composing with the multiplication, inverse and unit maps for $\Gamma$. In particular $\Gamma^{\alg}$ is always reduced. See \cite[\S\,2]{BL} for more information on the proalgebraic completion.

The proalgebraic completion of $\Gamma = \Z$ over an algebraically closed field of characteristic zero is well known: A representation of $\Z$ on a finite-dimensional vector space $V$ is given by an automorphism $\varphi$ of $V$. We may decompose $\varphi$ uniquely as a product $\varphi = \varphi_s \varphi_u$ where $\varphi_s$ is semisimple and $\varphi_u$ is unipotent with $\varphi_u \varphi_s = \varphi_s \varphi_u$. Unipotent automorphisms correspond to representations of $\Ge_{a,K}$. The automorphism $\varphi_s$ is determined up to conjugacy by its eigenvalues. Let $D$ be the diagonalizable group over $K$ corresponding to the abstract group $K^{\times}$i.e. $D = \spec K [K^{\times}]$ with co-multiplication
\[
 \Delta : K [K^{\times}] \longrightarrow K [K^{\times}] \otimes K [K^{\times}]
\]
being given by $\Delta (x) = x \otimes x$ for $x \in K^{\times}$. We have
\begin{equation} \label{eq:35}
 \Z^{\alg} = \Ge_a \times D
\end{equation}
as proalgebraic groups over $K$. Using the identification
\[
 D (K) = \Hom (K^{\times} , K^{\times})
\]
the map
\[
 \Z \longrightarrow \Z^{\alg} (K) = K \times \Hom (K^{\times} , K^{\times})
\]
sends $n$ to $(n , a \mapsto a^n)$.

The exact sequence, where $\mu_K$ are the roots of unity in $K^{\times}$,
\[
 1 \longrightarrow \mu_K \longrightarrow K^{\times} \longrightarrow K^{\times} / \mu_K \longrightarrow 1
\]
corresponds to the exact sequence of commutative group schemes over $K$
\[
 1 \longrightarrow D^0 \longrightarrow D \longrightarrow D^{\et} \longrightarrow 1 \; .
\]
Here $D^0 = \spec K [K^{\times} / \mu_K]$ is a pro-torus with character group $K^{\times} / \mu_K$ and $D^{\et} = \hZ /_K$ is the profinite completion of $\Z$ viewed as a pro-\'etale group scheme. Hence the connected \'etale sequence for $\Z^{\alg}_K$ reads as follows:
\[
 0 \longrightarrow \Ge_a \times D^0 \longrightarrow \Z^{\alg}_K \longrightarrow \hZ /_K \longrightarrow 0 \; .
\]
A similar description of $\Gamma^{\alg}$ can be given for any finitely generated abelian group $\Gamma$ instead of $\Z$.

For a connected topological space $X$ and a point $x \in X$ parallel transport along homotopy classes of paths in \eqref{eq:6} gave a homomorphism
\[
i : \pi_1 (X,x) \longrightarrow \pi_K (X,x) (K) \; .
\]

\begin{theorem}
\label{t41}
Let $X$ be a path-connected, locally path-connected and semi-locally simply connected space. Then for any $x \in X$ and any field $K$, the map $i$ induces an isomorphism of affine group schemes over $K$
\[
i^{\alg} : \pi_1 (X,x)^{\alg} \silo \pi_K (X,x) \; .
\]
The image under $i$ of $\pi_1 (X,x)$ is Zariski dense in $\pi_K (X,x)$. The maps $i$ and $i^{\alg}$ are functorial with respect to continuous maps of pointed spaces. The pseudo-torsor $P_X$ of \eqref{eq:25n} and the universal covering $\tX$ of $X$ are related by an isomorphism
\[
P_X = \tX \times^{\pi_1 (X,x)} \pi_K (X,x) (K) \; .
\]
In particular, $P_X$ is a torsor for the pro-discrete group $\pi_K (X,x) (K)$.
\end{theorem}

\begin{proof}
For the spaces $X$ in question, $\uFl_K (X)$ is $\otimes$-equivalent to $\uRep_K (\pi_1 (X,x))$ which implies the first assertions. The formula for $P_X$ and hence the fact that it is a torsor follows by going through the construction of $P_X$.
\end{proof}

For an arbitrary connected topological space $X$ one obtains quotients of $\pi_K (X,x)$ as follows. Assume that a discrete group $\Gamma$ acts by homeomorphisms on a connected topological space $Y$ such that $X = Y / \Gamma$ and such that every point $y \in Y$ has a neighborhood $U$ with $U \cap U \gamma = \emptyset$ for all $\gamma \neq e$. The proof of Proposition \ref{t31} applies also in this more general situation and gives an equivalence of categories
\[
\uRep_K (\Gamma) \silo \uFl_K (X) (\pi) \; , \; V \mapsto \pi^{\Gamma}_* (Y \times V) = Y \times^{\Gamma} V \; .
\]
Here $\pi : Y \to X$ is the projection. As before, by Corollary \ref{t23} the full subcategory $\uFl (X) (\pi)$ is closed under taking subobjects in $\uFl_K (X)$. The choice of a point $y \in Y$ over $x$ gives an identification of $V$ with $(Y \times^{\Gamma} V)_x$. Hence we get the following result:

\begin{prop}
\label{t42}
The preceding construction gives a faithfully flat morphism of affine group schemes over $K$ depending on $y \in Y$
\[
\pi_K (X,x) \twoheadrightarrow \Gamma^{\alg} \; .
\]
\end{prop}

One can define the projective limit
\[
\varprojlim_{(Y,y)} \Gamma^{\alg} = \varprojlim_{(Y,y)} \Aut (Y / X)^{\alg} \; .
\]
I do not know when the resulting faithfully flat morphism of affine groups over $K$
\begin{equation}
\label{eq:36}
\pi_K (X,x) \twoheadrightarrow \varprojlim_{(Y,y)} \Gamma^{\alg}
\end{equation}
is also a closed immersion and therefore an isomorphism. This is the case if and only if every flat bundle $E$ on $X$ is a subquotient of a bundle of the form $Y \times^{\Gamma} V$ above. This is true for locally connected topological spaces $X$ by Proposition \ref{t22}, b). In general there is a $\Gamma$-torsor $Y$ with $E = Y \times^{\Gamma} V$ by Theorems \ref{t24n} and \ref{t25n} but it may be disconnected. For the topological spaces in Theorem \ref{t41} the condition is satisfied and \eqref{eq:36} is an isomorphism. In fact, the universal covering of $X$ dominates all other coverings and we see again that there is an isomorphism
\[
\pi_K (X,x) \silo \pi_1 (X,x)^{\alg} \; .
\]
We can calculate $\pi_K (X,x)$ for some solenoids using the following continuity property.

\begin{theorem}
\label{t43}
Let $\Lambda$ be a directed partially ordered set and $(X_{\lambda} , p_{\lambda \mu})$ a projective system indexed by $\Lambda$ of compact connected Hausdorff spaces and continuous maps. Fix a point $x$ of the compact connected Hausdorff space $X = \varprojlim_{\lambda \in \Lambda} X_{\lambda}$ and set $x_{\lambda} = p_{\lambda} (x)$ where $p_{\lambda} : X \to X_{\lambda}$ is the projection map. Let $K$ be a field. Then the morphisms of group schemes over $K$
\[
p_{\lambda*} : \pi_K (X,x) \longrightarrow \pi_K (X_{\lambda} , x_{\lambda}) \quad \text{for} \; \lambda \in \Lambda
\]
induce an isomorphism
\[
\pi_K (X,x) \silo \varprojlim_{\lambda \in \Lambda} \pi_K (X_{\lambda} , x_{\lambda}) \; .
\]
\end{theorem}

\begin{proof}
We have to show that
\[
\uFl_K (X) = \varinjlim_{\lambda \in \Lambda} \uFl_K (X_{\lambda}) \; .
\]
This means that flat bundles on $X$ and morphisms between them are obtained via $p^*_{\lambda}$ from bundles and morphisms on the level of $X_{\lambda}$ for some $\lambda \in \Lambda$. We show this for bundles. The proof for morphisms is similar. An analogous assertion for finite coverings instead of flat bundles is given in \cite{KS} Proposition 2.11. Following \cite{KS} we call a subset of $X$ basis-open if it is of the form $p^{-1}_{\lambda} (U)$ for some $\lambda \in \Lambda$ and some open $U \subset X_{\lambda}$. Since $\Lambda$ is directed, the basis-open subsets form a basis of the topology of $X$. In the proof of descent for bundles we will use the following four facts {\bf (A)}--{\bf (D)}.

{\bf (A)} For $X = \varprojlim_{\lambda} X_{\lambda}$ as in the theorem, consider a subspace $Y_{\mu} \subset X_{\mu}$ for some $\mu \in \Lambda$. For $\lambda \ge \mu$ set $Y_{\lambda} = p^{-1}_{\lambda \mu} (Y_{\mu})$. Then the projective limit topology on $Y = \varprojlim_{\lambda} Y_{\lambda} \subset X$ equals the subspace topology of $Y$ in $X$.

This holds because a basis of the projective limit topology of $Y$ is given by the sets
\[
(p_{\lambda} |_Y)^{-1} (Y_{\lambda} \cap O_{\lambda}) = Y \cap p^{-1}_{\lambda} (Y_{\lambda}) \cap p^{-1}_{\lambda} (O_{\lambda}) = Y \cap p^{-1}_{\lambda} (O_{\lambda})
\]
where $\lambda \ge \mu$ and $O_{\lambda} \subset X_{\lambda}$ is open.

{\bf (B)} For $X = \varprojlim_{\lambda} X_{\lambda}$ as in the theorem, let $U^1_{\mu} , \ldots , U^n_{\mu}$ be open sets in $X_{\mu}$ such that $p^{-1}_{\mu} (U^1_{\mu}) , \ldots , p^{-1}_{\mu} (U^n_{\mu})$ are a cover of $X$. Then there is some $\lambda \ge \mu$ such that the open sets $U^i_{\lambda} =  p^{-1}_{\lambda \mu} (U^i_{\mu})$ form a cover of $X_{\lambda}$. Moreover this remains true for the pullbacks to $X_{\lambda'}$ for any $\lambda' \ge \lambda$.

Write $O = U^1_{\mu} \cup \ldots \cup U^n_{\mu}$. Then $O$ is open in $X_{\mu}$ and $p^{-1}_{\mu} (O) = X$. We have
\[
\emptyset = X \setminus p^{-1}_{\mu} (O) = \varprojlim_{\lambda \ge \mu} (X_{\lambda} \setminus p^{-1}_{\lambda \mu} (O)) \; .
\]
The sets $X_{\lambda} \setminus p^{-1}_{\lambda \mu} (O)$ being compact and $\Lambda$ directed, it follows that $X_{\lambda} \setminus p^{-1}_{\lambda\mu} (O) = \emptyset$ for some $\lambda \ge \mu$ and hence $X_{\lambda'} = p^{-1}_{\lambda' \mu} (O)$ for all $\lambda' \ge \lambda$.

{\bf (C)} For $X = \varprojlim_{\lambda} X_{\lambda}$ as in the theorem, let $g : X \to K$ be a locally constant function. Then there exists an index $\mu \in \Lambda$ and a locally constant function $g_{\mu} : X_{\mu} \to K$ such that $g = g_{\mu} \verk p_{\mu}$. 

{\bf (D)} In {\bf (C)} if $g = g_{\mu_1} \verk p_{\mu_1} = g_{\mu_2} \verk p_{\mu_2}$ then for some $\lambda \ge \mu_1 , \lambda \ge \mu_2$ we have $g_{\mu_1} \verk p_{\lambda \mu_1} = g_{\mu_2} \verk p_{\lambda \mu_2}$. The same is true for any $\lambda' \ge \lambda$. 

Assertions {\bf (C)} and {\bf (D)} are equivalent to the formula
\[
H^0 (X, \uK) = \varinjlim_{\lambda} H^0 (X_{\lambda} , \uK ) \; .
\]
This is a special case of \cite{B} Lemma 14.2 or Corollary 14.6. 

Let $E$ be a flat bundle on $X$ and choose a finite trivializing atlas for $E$ whose open sets $U_1 , \ldots , U_n$ are basis-open. Let $(g_{ij})$ be the corresponding \v{C}ech cocycle. Since $\Lambda$ is directed there is some $\mu \in \Lambda$ with $U_i = p^{-1}_{\mu} (U^{\mu}_i)$ for all $i$ where $U^{\mu}_i$ is open in $X_{\mu}$. By assertion {\bf (A)} we may assume that $U^{\mu}_1 , \ldots , U^{\mu}_n$ are a cover of $X_{\mu}$. Choose a cover $V^{\mu}_1 , \ldots , V^{\mu}_n$ of $X_{\mu}$ by open subsets with $V^{\mu}_i \subset A^{\mu}_i \subset U^{\mu}_i$ where $A^{\mu}_i$ is the closure of $V^{\mu}_i$ in $X_{\mu}$ and hence compact. Set $V_i = p^{-1}_{\mu} (V^{\mu}_i)$ and $A_i = p^{-1}_{\mu} (A^{\mu}_i)$. Then we have $V_i \subset A_i \subset U_i$ and $\{ V_i \}$ resp. $\{ A_i \}$ are open resp. closed covers of $X$. We have
\[
A_i \cap A_j = \varprojlim_{\lambda \ge \mu} p^{-1}_{\lambda} (A^{\lambda}_i \cap A^{\lambda}_j)
\]
where $A^{\lambda}_i = p^{-1}_{\lambda \mu} (A^{\mu}_i)$ for all $i , \lambda \ge \mu$. Applying {\bf (A)} and {\bf (C)} to this projective system (and the component functions of $g_{ij} \, |_{A_i \cap A_j}$) it follows that there are locally constant functions $g^{\nu}_{ij} : A^{\nu} _i \cap A^{\nu}_j \to \GL_r (K)$ for some index $\nu \ge \mu$ such that we have
\[
g_{ij} \, |_{A_i \cap A_j} = g^{\nu}_{ij} \verk p_{\nu} \; .
\]
The cocycle condition for the $g_{ij}$ gives the equations
\[
g_{ij} (x) \verk g_{jk} (x) = g_{ij} (x) \quad \text{for all} \; x \in A_i \cap A_j \cap A_k \; .
\]
Applying {\bf (D)} to the projective system
\[
A_i \cap A_j \cap A_k = \varprojlim_{\lambda \ge \nu} A^{\lambda}_i \cap A^{\lambda}_j \cap A^{\lambda}_k
\]
it follows that from some index $\lambda \ge \nu$ on, the locally-constant functions $g^{\lambda}_{ij} = g^{\nu}_{ij} \verk p_{\lambda \nu}$ on $A^{\lambda}_i \cap A^{\lambda}_j$ satisfy the cocycle condition. Restricting the $g^{\lambda}_{ij}$ to $V^{\lambda}_i \cap V^{\lambda}_j$ we obtain a cocycle $(g^{\lambda}_{ij})$ of locally constant $\GL_r (K)$-valued functions on the open cover $\{ V^{\lambda}_i \}$ of $X_{\lambda}$. It defines a vector bundle $E_{\lambda}$ in $\uFl_K (X_{\lambda})$ together with a canonical isomorphism $p^*_{\lambda} E_{\lambda} \cong E$ in $\uFl_K (X)$.
\end{proof}

\begin{exmp}
Fix a set $P$ of prime numbers and let $\Lambda_P$ be the set of positive integers whose prime factors belong to $\Lambda_P$. Writing $\mu \le \lambda$ if $\mu$ divides $\lambda$ the set $\Lambda_P$ becomes a directed poset. For $\lambda \in \Lambda_P$ set $X_{\lambda} = \R / \Z$ and for $\lambda \ge \mu$ let $p_{\lambda \mu} : X_{\lambda} \to X_{\mu}$ be the multiplication by $\lambda / \mu$. The projective limit is a solenoid
\[
\Sa_P = \varprojlim_{\lambda \in \Lambda_P} X_{\lambda} = \R \times^{\Z} \hZ_P \; ,
\]
where $\hZ_P = \prod_{p \in P} \Z_p$. Let $0$ be the zero element of the compact connected topological group $\Sa_P$. By Theorems \ref{t41} and \ref{t43} we have
\[
\pi_K (\Sa_P , 0) = \varprojlim_{\lambda \in \Lambda_P} \pi_K (\R / \Z, 0) = \varprojlim_{\lambda \in \Lambda_P} \Z^{\alg}_K \; .
\]
Here the transition map from the $\lambda$-th copy of $\Z^{\alg}_K$ to the $\mu$-th copy for $\lambda \ge \mu$ is given by multiplication with $\lambda / \mu$. This holds because $N$-multiplication on $\R / \Z$ induces $N$-multiplication on $\pi_1 (\R / \Z , 0) = \Z$ and hence also on $\pi_K (\R / \Z , 0) = \pi_1 (\R / \Z , 0)^{\alg}_K$. For simplicity we now assume that $K$ is algebraically closed of characterstic zero. Recall the decomposition $\Z^{\alg}_K = \Ge_a \times D$ of \eqref{eq:35}. For any positive integer $N$, the $N$-multiplication on $\Ge_a$ is an isomorphism and hence
\[
\varprojlim_{\lambda \in \Lambda_P} \Ge_a = \Ge_a \; .
\]
\end{exmp}

On $D = \spec K [K^{\times}]$ the $N$-multiplication map comes from the $N$-th power map on $K^{\times}$. Hence we have
\[
\D_P := \varprojlim_{\lambda \in \Lambda_P} D = \spec K \big[ \varinjlim_{\lambda \in \Lambda_P} K^{\times} \big] = \spec K [K^{\times} / \mu_{P^{\infty}}] \; ,
\]
where
\[
\mu_{P^{\infty}} = \{ \zeta \in K^{\times} \mid \zeta^{\lambda} = 1 \quad \text{for some} \; \lambda \in \Lambda_P \} \; .
\]
Thus we find:
\[
\pi_K (\Sa_P , 0) = \Ge_a \times \D_P \; .
\]
The connected component $\D^0_P$ of $\D_P$ is independent of $P$. It is the pro-torus
\[
\D^0 = \spec K [K^{\times} / \mu_K ] 
\]
with character group $K^{\times} / \mu_K$. The maximal pro-\'etale quotient of $\D$ is the pro-\'etale group scheme with character group $\mu_K / \mu_{P^{\infty}}$. It is isomorphic to $\prod_{p \notin P} \Z_p$ viewed as a pro-finite-constant group scheme over $K$. Hence we have
\[
\pi_K (\Sa_P , 0)^0 = \Ge_a \times \D^0 \quad \text{and} \quad \pi_K (\Sa_P , 0)^{\et} = \prod_{p \notin P} \Z_p /_K \; .
\]
Via the isomorphism \eqref{eq:27}
\[
\pi^{\et}_1 (\Sa_P , 0) /_K \silo \pi_K (\Sa_P , 0)^{\et} = \prod_{p \notin P} \Z_p /_K \; ,
\]
we see that the more multiplications by primes are inverted on $S^1$ by passing to the solenoid $\Sa_P$, the fewer finite coverings remain. For the full solenoid $\Sa$ where $P$ consists of all prime numbers we have
\[
\pi^{\et}_1 (\Sa , 0)_K = \pi_K (\Sa , 0)^{\et} = 0 \; .
\]
Hence
\[
\pi_K (\Sa , 0) = \Ge_a \times \D^0 
\]
is connected, where $\D^0$ is the pro-torus with character group $K^{\times} / \mu_K$.

\begin{rem}
The \v{C}ech fundamental group $\check{\pi}_1$ is continuous and hence we have
\[
\check{\pi}_1 (\Sa_P , 0) = \varprojlim_{\lambda} \check{\pi}_1 (\R / \Z , 0) = \varprojlim_{\lambda} \Z \; .
\]
Here the transition maps are multiplication by $\lambda / \mu$ as before. It follows that $\check{\pi}_1 (\Sa_P , 0)$ is trivial. 
\end{rem}

We now relate the groups $\pi_K (X, x)$ to cohomology. For a topological space $X$ let $H^i (X , \Fh)$ denote the derived functor cohomology of a sheaf of abelian groups $\Fh$. For $i = 0 , 1$ it is isomorphic to \v{C}ech cohomology. For a sheaf of possibly non-abelian groups $\Gh$ we consider the \v{C}ech cohomology set $\check{H}^1 (X , \Gh)$.

\begin{prop}
\label{t44}
Let $X$ be a connected topological space, $x \in X$ and $K$ a field, $r \ge 1$. \\
a) There is a canonical isomorphism
\[
\Hom (\pi_K (X,x) , \GL_{r / K}) / \GL_r (K) \silo H^1 (X , \uGL_r (K)) \; .
\]
Here $\GL_r (K)$ acts by conjugation on the group scheme $\GL_{r / K}$. In particular
\[
\Hom (\pi_K (X,x) , \Ge_m) \silo H^1 (X , \uK^{\times}) \; .
\]
b) There is a canonical isomorphism
\[
\Hom (\pi_K (X,x) , \Ge_a) \silo H^1 (X, \uK) \; .
\]
c) The above maps are functorial with respect to base point preserving continuous maps of connected topological spaces.
\end{prop}

\begin{proof}
a) Both sides describe isomorphism classes of flat vector bundles of rank $r$ on $X$.\\
b) We have an isomorphism where $\pi_K (X,x)$ acts trivially on $\A_K = \A^1_K$
\begin{equation}
\label{eq:41}
\Hom (\pi_K (X,x) , \Ge_a) = \Ext^1_{\pi_K (X,x)} (\A_K , \A_K) \; .
\end{equation}
Namely, we view $\Ge_a$ as the algebraic group $U_2$ of unipotent matrices $\left( \begin{smallmatrix} 1 & * \\ 0 & 1 \end{smallmatrix} \right)$. Given a homomorphism $\lambda : \pi_K (X,x) \to \Ge_a$ we let $\pi_K (X,x)$ act on $\A^2_K$ via $U_2$. This defines an extension of $\A_K$ by itself in $\uRep_K (\pi_K (X,x))$. The resulting map is the isomorphism \eqref{eq:41}. By the equivalence of categories
\[
\uRep_K (\pi_K (X,x)) \silo \uLoc_K (X)
\]
we obtain an isomorphism:
\[
\Hom (\pi_K (X,x) , \Ge_a) \silo \Ext^1_{\uLoc_K (X)} (\uK , \uK) \overset{(!)}{=} \Ext^1_X (\uK , \uK) = H^1 ( X , \uK) \; .
\]
Note here that any extension of sheaves of $K$-vector spaces
\[
0 \longrightarrow \uK \longrightarrow \Fh \longrightarrow \uK \longrightarrow 0 
\]
splits locally, so that $\Fh$ is in $\uLoc_K (X)$. This holds because for all $y \in X$ we have $\varinjlim_{U \ni y} H^1 (U , \uK) = 0$, a special case of \cite{G} 3.8.2 Lemma. For another proof of b) we could argue that $H^1 (X , \uK) = \check{H}^1 (X , U_2 (K))$ classifies the isomorphism classes of rank $2$ unipotent flat bundles, i.e. extensions in $\uFl_K (X)$ of $\uK$ by itself, and argue as above. Here the lemma from \cite{G} is implicitely used in the identification of $H^1$ with \v{C}ech $\check{H}^1$. 
\end{proof}

\begin{rems}
a) It is instructive to apply the proposition to the solenoids $\Sa_P$.\\
b) Cohomology is a functor under {\it arbitrary} continuous maps. This is not clear to me for the left hand sides of the isomorphisms in the proposition for general $K$. Namely, I do not know if the conjugacies coming from the isomorphisms of the fibre functors in different points are in general already defined over $K$. This is the case if $K$ is algebraically closed, by Deligne's general result on fibre functors \cite{D}, \cite{Cou}, \cite{Wib}.
\end{rems}

There is a canonical exact sequence
\[
1 \longrightarrow \pi^u_K (X,x) \longrightarrow \pi_K (X,x) \longrightarrow \pi^{\red}_K (X,x) \longrightarrow 1 \; .
\]
Here $\pi^{\red}_K (X,x)$ is the maximal pro-reductive quotient of $\pi_K (X,x)$ and $\pi^u_K (X,x)$ the pro-unipotent radical, a connected group scheme. For a field extension $L / K$ we have a natural morphism
\begin{equation}
\label{eq:42}
\pi_L (X,x) \longrightarrow \pi_K (X,x) \otimes_K L 
\end{equation}
of group schemes over $L$ and similarly for $\pi^u_K$ and $\pi^{\red}_K$. The example of $X = S^1$ already shows that $\pi^{\red}_K$ and hence $\pi_K$ itself do not commute with the base change $L / K$, i.e. \eqref{eq:42} is not an isomorphism. However, in the example $X = S^1$, the unipotent part $\pi^u_K (S^1 , x) = \Ge_{a,K}$ does commute with base change and this may be true in general. 

The $\Ext$-groups in $\uFl_K (X)$ for $\car K = 0$ may be expressed in terms of $\pi_K = \pi_K (X,x)$-cohomology by the standard Tannakian formalism, c.f. \cite{Ja} Appendix C4. Namely, for flat vector bundles $F$ and $E$ corresponding to $\pi_K (X,x)$-representations on $F_x$ and $E_x$ we have
\[
\Ext^i_{\uFl_K (X)} (F,E) = H^i (\pi_K , \uHom (F_x , E_x)) \; .
\]
Here $H^i$ is the (colimit of the) cohomology theory introduced in \cite{H}. Moreover, for a $\pi_K (X,x)$-representation on a finite dimensional $K$-vector space $V_x$, we have
\[
H^i (\pi_K , V_x) = H^i (\pi^u_K , V_x)^{\pi^{\red}_K}
\]
and
\[
H^i (\pi^u_K , V_x) = H^i (\Lie \pi^u_K , V_x)^{\pi^{\red}_K} \; .
\]
Setting $V = \uHom (F,E)$ in $\uFl_K (X)$, we therefore obtain
\[
\Ext^i_{\uFl_K (X)} (F, E) = H^i (\Lie \pi^u_K , V_x)^{\pi^{\red}_K} \; .
\]
The full subcategory $\uLoc_K (X)$ of the category of all sheaves of $K$-vector spaces on $X$ is stable under extensions. Hence for $F = \uK$ and $i = 0,1$ we get an isomorphism
\[
H^i (X,E) = H^i (\Lie \pi^u_K , E_x)^{\pi^{\red}_K} \; .
\]
For $i = 2$ the group on the left contains the one on the right.

We end with a couple of open issues: 

Does $\pi^u_K (X,x)$ commute with base change to fields $L / K$? 

For pointed connected spaces $(X,x) , (X' , x')$ consider the natural faithfully flat map
\begin{equation}
\label{eq:43}
\pi_K (X \times X' , (x,x')) \longrightarrow \pi_K (X,x) \times_K \pi_K (X',x')
\end{equation}
induced by the projections. Note that \eqref{eq:43} is split by the product of the morphisms induced by the maps $X \to X \times X' , z \mapsto (z, x')$ and $X' \to X \times X' , z' \mapsto (x,z')$. In what generality are \eqref{eq:43} and its variants for $\pi^u_K$ and $\pi^{\red}_K$ isomorphisms? 

We introduced fundamental groupoids in our setting with a view towards proving Seifert van Kampen results, at least for $\pi^u_K$. According to \cite{L} there exist free products for pro-unipotent algebraic group schemes. One would also need amalgamated products.

\end{document}